\RequirePackage{etoolbox}
\csdef{input@path}{{style/}{graphics/}}
\documentclass[aop,MSNbibl,seceqn,dvips]{arximspdf}
\usepackage{mathrsfs}
\usepackage{graphicx}

\doi{10.1214/13-AOP894} 
\volume{43}
\issue{2}
\pubyear{2015}
\firstpage{738}
\lastpage{781}

\makeatletter
\newcommand{\rrVert}{\Vert}
\newcommand{\rrvert}{\vert}
\newcommand{\llVert}{\Vert}
\newcommand{\llvert}{\vert}
\newcommand{\binom}[2]{\pmatrix{#1\cr #2}}
\newcommand{\llbracket}{[\![}
\newcommand{\rrbracket}{]\!]}
\newtheorem{theorem}{Theorem}[section]
\newtheorem{lem}[theorem]{Lemma}
\newtheorem{prop}[theorem]{Proposition}
\newtheorem{cor}[theorem]{Corollary}
\newproclaim{rem}{Remark}
\makeatother

\begin{document}
\begin{frontmatter}

\title{The dual tree of a recursive triangulation of~the~disk\thanksref{T1}}
\runtitle{The dual tree of a recursive triangulation of the disk}

\begin{aug}
\author[A]{\fnms{Nicolas} \snm{Broutin}\corref{}\ead[label=e1]{nicolas.broutin@inria.fr}}
\and
\author[B]{\fnms{Henning} \snm{Sulzbach}\ead[label=e2]{henning.sulzbach@gmail.com}}
\runauthor{N. Broutin and H. Sulzbach}
\affiliation{INRIA Paris-Rocquencourt and McGill University}
\address[A]{INRIA Paris-Rocquencourt\\
Domaine de Voluceau\\
78153 Le Chesnay\\
France\\
\printead{e1}} 
\address[B]{School of Computer Science\\
McGill University\\
3480 University Street\\
Montreal\\
Canada H3A 0E9\\
\printead{e2}}
\end{aug}
\thankstext{T1}{Supported by MAEE and MESR Procope Grant 23133PG (MAEE and MESR) and DAAD Procope Project 50085686.}

\received{\smonth{11} \syear{2012}}
\revised{\smonth{10} \syear{2013}}

%
\begin{abstract}
In the recursive lamination of the disk, one tries to add chords one after
another at random; a chord is kept and inserted if it does not
intersect any of the
previously inserted ones. Curien and Le Gall [\textit{Ann. Probab.} \textbf{39} (2011)
2224--2270] have proved that the set of chords converges to a limit triangulation
of the disk encoded
by a continuous process $\mathscr M$. Based on a new approach
resembling ideas from the
so-called contraction method in function spaces, we prove that, when
properly rescaled,
the planar dual of the discrete lamination converges almost surely in
the Gromov--Hausdorff
sense to a limit real tree $\mathscr T$, which is encoded by $\mathscr M$.
This confirms a conjecture of Curien and Le Gall.
\end{abstract}

%

\end{frontmatter}

\section{Introduction and main results}\label{secintro}
In \cite{legalcu}, Curien and Le Gall introduce the model of \textit{random recursive
triangulations} of the disk. The construction goes as follows:
at $n=1$, two points are sampled independently with uniform
distribution on the circle.
They are connected by a chord (a straight line) which
splits the disk into two fragments. Later on, at each step, two
independent points are
sampled uniformly at
random on the circle and are connected by a chord if the latter does
not intersect any of
the previously inserted chords; in other words the two points are
connected by a chord if
they both fall in the same fragment. This gives rise to a sequence of
\emph{laminations} of
the disk; for us a lamination will be a collection of chords which may
only intersect at
their end points. At time $n$, the lamination $\mathfrak L_n$ consists
of the
union of the chords
inserted up to time $n$. As an increasing closed subset of the disk,
$\mathfrak L_n$ converges,
it is proved in \cite{legalcu} that
\[
\mathfrak L_\infty=\overline{\bigcup_{n\ge1}
\mathfrak L_n}
\]
is a triangulation of the disk in the sense that any face of the
complement is an open
triangle whose vertices lie on the circumference of the circle (see
\cite{al94circle}).
Curien and Le~Gall \cite{legalcu} then study thoroughly the limit
triangulation $\mathfrak L_\infty$;
in particular, they compute the Hausdorff dimension of $\mathfrak
L_\infty$ using
a representation of the limit based on an encoding by a random function.
The main purpose of the present paper is to study the tree that is dual
to the lamination (as in planar dual). In particular, we prove that, seen
as a metric space equipped with the graph distance suitably rescaled,
the planar dual of
the lamination $\mathfrak L_n$ converges almost surely to a nondegenerate
random metric space $\mathscr T$, hence confirming a conjecture of
Curien and
Le~Gall \cite{legalcu}.
Before we go further in our description or our results and approach
(Section~\ref{secmainresults}), we introduce the relevant notation
and terminology.

\subsection{Laminations, dual trees, encoding functions and convergence}

\subsubsection*{Setting and notation}
We consider the disk $\mathscr D:=\{z\in\mathbb{C}\dvtx  2\pi|z|\le1\}$
and the circle
$\mathscr C= \{z \in\mathbb{C}\dvtx  2 \pi|z|= 1 \}$ as subsets of the
complex plane.
For convenience, the circle $\mathscr C$ is identified with the unit interval
where the
points $0$ and $1$ have been glued: we identify $s \in[0,1]$
with the point $\frac{1}{2 \pi} \exp(2 \pi i s) \in\mathscr C$.
Accordingly, we let $\llbracket x,y \rrbracket$ denote the (closed)
straight chord
joining the two points
of $\mathscr C$ corresponding to $x,y\in[0,1]$, $x<y$.
At some time $n$, we let $\mathfrak L_n$ be the collection of inserted chords.
The set
$\mathscr D\setminus\mathfrak L_n$ consists of a number of connected
components that
we call
\emph{fragments}; the \emph{mass} of a given fragment is the Lebesgue measure
of its intersection with the circle $\mathscr C$.

\subsubsection*{The lamination encoded by a function}
The key to studying the lamination $\mathfrak L_n$ and its dual tree is an
encoding by a
function, as in the pioneering work by Aldous~\cite{al94circle,al94}.
Let $C_n(s)$ denote the
number of chords in $\mathfrak L_n$ which intersect the straight line going
through the points
$0$ and $s$ of the circle. A priori, for any \mbox{$n \geq1$}, $C_n(s)$ is
not properly
defined at endpoints of chords, and we fix this issue by considering it
as a
right-continuous step function.
This convention enables us to regard every relevant process on the unit
interval throughout the paper as c{\`a}dl{\`a}g (right-continuous with
left-limits) and continuous at $1$, and we will do so. 
The function~$C_n$ encodes the lamination $\mathfrak L_n$ in the
following sense.

For a function $f\dvtx  [0,1] \to[0, \infty)$ with $f(0) = f(1) = 0$ having
c{\`a}dl{\`a}g
paths, one defines a lamination $\mathfrak L_f$ as follows. Given
$x,y\in
(0,1)$, with $x<y$,
the chord $\llbracket x,y \rrbracket$ is said to be \emph{compatible}
with $f$, or
$f$-compatible, if there
exists $w$ such that
\[
\forall s\in(x,y)\qquad\mbox{one has }f(s)>w\quad\mbox{and}\quad \max\bigl\{
f(x-), f(y)\bigr\}\le w.
\]
Then we define the lamination $\mathfrak L_f$ as the smallest compact subset
of the disk which
contains all the chords which are compatible with $f$. ($\mathfrak L_f$
is the
set of chords which
are either compatible with $f$, or the limit of compatible chords for
the Hausdorff metric.)
This definition is consistent with the ones in \cite{LePa2008a,legalcu,Ko2011b} for the
case of continuous excursions. Then the
\emph{height processes} $(C_n)_{n\ge1}$ encodes the laminations
$(\mathfrak L
_n)_{n\ge1}$
in the sense that $\mathfrak L_n=\mathfrak L_{C_n}$ for every $n\ge1$.
Laminations
are seen
as compact subsets of the disk $\mathscr D$, and we use the
Hausdorff distance $\mathrm{d_{\mathrm{H}}}$ to compare them. To fix
the notation, recall
that for
two compact subsets $A$ and $B$ of $\mathscr D$, we have
\[
\mathrm{d_{\mathrm{H}}}(A,B) = \inf\bigl\{\varepsilon> 0\dvtx  A^\varepsilon
\subseteq B, B^\varepsilon \subseteq A \bigr\},
\]
where $A^\varepsilon= \{x \in\mathscr D\dvtx  |x-a| < \varepsilon$ for some $a \in A \}$.

%
\begin{figure}[b]

\includegraphics{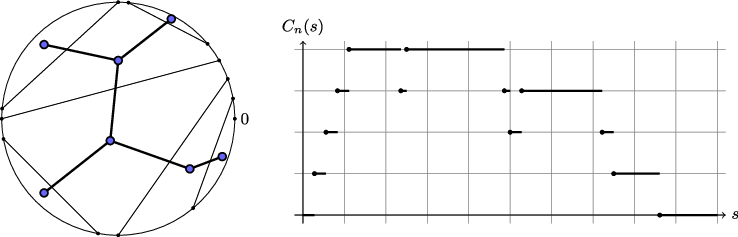}

\caption{A lamination, its right-continuous height process and the
corresponding rooted dual tree. Distances in the tree correspond to the
number of chords separating fragments in the lamination.}\label{figlam-dual}
\end{figure}

\subsubsection*{Trees encoded by functions}
Our main concern is the tree $T_n$ that is dual to the lamination
$\mathfrak L
_n$, and
its scaling limit as $n\to\infty$. Each fragment in $\mathfrak L_n$ is
associated with a node
and two nodes $u$ and $v$ are connected in $T_n$ in the tree if and
only if
the corresponding fragments $S_n(u)$ and $S_n(v)$ share a chord $\ell$
of $\mathfrak L_n$
[more precisely $S_n(u)\cup S_n(v)\cup\ell$ is a connected component
of $\mathscr D$].
Let $d_n$ be the graph distance in $T_n$, which
comes with a natural distinguished point---the root---the fragment whose
intersection with the circle $\mathscr C$ contains the point $0$.
The encoding of laminations by functions turns out to also encode the
dual tree.
The value of the encoding function $C_n(s)$ at a given point $s\in
[0,1]$ is precisely
the height in $T_n$ (distance to the root) of the node corresponding to
the face whose
intersection with the circle contains the point $s$; see Figure~\ref{figlam-dual}.
More precisely, the function $C_n$
actually encodes the metric structure of the dual tree $T_n$ in the
following sense
\mbox{\cite{DuLe2005a,Aldous1993a,evans,LeGall2005}}. Consider a {c\`adl\`ag} function
$f\dvtx [0,1]\to[0,\infty)$ such that $f(0)=f(1)=0$ and $f(s)>0$ for all
$s\in(0,1)$.
Define $d_f:=[0,1]^2\to[0,\infty)$ by
\[
d_f(x,y)=f(x) + f(y) - 2 \inf\bigl\{f(s)\dvtx  x\wedge y \le s \le x
\vee y \bigr\}.
\]
One easily verifies that $d_f$ is a pseudo-metric on $[0,1]$. Let
$x\sim y$ if $d_f(x,y)=0$.
Write $\mathcal T_f$ for the quotient $[0,1]/ \sim$ and consider the metric
space $(\mathcal T_f, d_f)$. Then $(\mathcal T_{C_n}, d_{C_n})$ is
isometric to the
dual tree $(T_n,d_n)$.

\subsubsection*{Real trees and Gromov--Hausdorff convergence} The natural scaling limits for large trees are \emph{real trees},
which are
encoded by continuous functions.
A compact metric space $(X,d)$ is called a real tree if it is geodesic
and acyclic:
\begin{itemize}
\item for every $x,y\in X$ there exists a unique isometry $\phi
_{x,y}\dvtx [0,d(x,y)]\to X$ such that $\phi_{xy}(0)=x$ and $\phi
_{xy}(d(x,y))=y$, and
\item if $q$ is a continuous injective map from $[0,1]$ to $X$ such
that $q(0)=x$ and $q(1)=y$ then $q([0,1])=\phi_{x,y}([0,d(x,y)])$.
\end{itemize}
For a {c\`adl\`ag} function $f$ with the properties above, the metric
space $(\mathcal T_f, d_f)$ is a~real
tree.

Given two compact metric spaces $(X,d)$ and $(X',d')$, one defines the
Gromov--Hausdorff
distance $\mathrm{d_{\mathrm{GH}}}(X,X')$ between $X$ and $X'$ to be
the infimum of all quantities
$\delta_{\mathrm H}(\phi(X),\phi'(X'))$ ranging over the choice of
compact metric
spaces $(Z,\delta)$, and isometries $\phi\dvtx X\to Z$ and $\phi'\dvtx X'\to
Z$, where
$\delta_{\mathrm H}$ denotes the Hausdorff distance in $Z$. The
distance $\mathrm{d_{\mathrm{GH}}}$ is a
pseudo-metric between compact metric spaces, and induces a metric on
the quotient space
which identifies two metric spaces if they are isometric, see, for example,
\cite{LeGall2005,evans,Gromov1999}.

Comparing Hausdorff convergence of laminations and Gromov--Hausdorff
convergence of
their dual trees, the trees (or the height processes) are arguably the important
objects: convergence of a sequence of increasing laminations as a
subset of the
complex plane only concerns the set of inserted chords, the time-scale
and order in
which they are inserted is completely irrelevant. Conversely,
convergence of the
(rescaled) height function implies convergence of the lamination under
suitable mild
additional assumptions; see Section~\ref{secconvlam}.

\subsection{Main results and general approach}\label{secmainresults}

Using the theory of fragmentation processes \cite{Bertoin2006}, Curien
and Le~Gall \cite{legalcu} prove
that there exists a random continuous process $\mathscr M$ which
encodes the
limiting triangulation in the sense that $\mathfrak L_\infty$
is distributed like $\mathfrak L_{\mathscr M}$.
For this, they prove pointwise convergence of the encoding functions:
for every $s\in[0,1]$, we have $n^{-\beta/2} C_n(s) \to\mathscr M(s)$
in probability, as $n\to\infty$, where the constant $\beta$ is given by
%
\begin{equation}
\label{eqdef-beta} \beta= \frac{\sqrt{17}-3}{2}=0.561552\ldots.
\end{equation}
They also show that for any $\varepsilon> 0$, almost surely, the process
$\mathscr M$ is
$(\beta-\varepsilon)$-H\"older continuous and for any $s\in[0,1]$ we have
%
\begin{equation}
\label{eqfree-const} \mathbf{E} \bigl[\mathscr M(s)\bigr] = \kappa\bigl(s(1-s)
\bigr)^\beta
\end{equation}
for some constant $\kappa> 0$ which was not identified in \cite{legalcu}.
Finally, the random function $\mathscr M$ inherits the recursive
structure of
the lamination
process and satisfies the
following distributional fixed-point equation: let $\mathscr M^{(0)},
\mathscr M
^{(1)}$ denote independent
copies of $\mathscr M$, let also $(U,V)$ be independent of $(\mathscr
M^{(0)}, \mathscr M
^{(1)})$ with density
$2 \mathbf{1}_{ \{ 0 \leq u \leq v \leq1 \} }$ on $[0,1]^2$. Then the
process defined by
%
\begin{equation}
\label{eqfixchord} 
\qquad \cases{ \displaystyle\bigl(1-(V-U)
\bigr)^\beta\mathscr M^{(0)} \biggl( \frac
{s}{1-(V-U)}\biggr),\qquad \mbox{if $s<U$},
\vspace*{5pt}\cr
\displaystyle\bigl(1-(V-U)\bigr)^\beta
\mathscr M^{(0)} \biggl( \frac{U}{1-(V-U)} \biggr)+(V-U)^\beta\mathscr M^{(1)} \biggl( \frac{s-U}{V-U} \biggr),
\vspace*{3pt}\cr
\hspace*{195.5pt} \mbox{if $U\le s<V$},
\vspace*{5pt}\cr
\displaystyle\bigl(1-(V-U)\bigr)^\beta\mathscr M^{(0)}
\biggl( \frac{s-(V-U)}{1-(V-U)} \biggr),\qquad\mbox{if $s\ge V$,}}
\end{equation}
is distributed like the initial process $\mathscr M$.

%
\begin{rem*}
Note in passing that the constant $\beta$ defined in (\ref{eqdef-beta}) appears in several contexts, such as the Hausdorff
dimension of the standard random Cantor set, in the
problem of parking arcs on the circle \cite{coffmall,bargne}, in the
analysis of the complexity of partial match retrieval algorithms
in search trees \cite{FlPu1986,FlGoPuRo1993,cujo11,BrNeSu12} or in
models from biological physics \cite{phy}.
\end{rem*}
We prove that the convergence of $n^{-\beta/2}C_n$ to $\mathscr M$ is actually
uniform with probability one. For any c{\`a}dl{\`a}g or continuous
function $f$,
we denote its supremum by~$\| f \|$.

\begin{theorem}\label{thmmain}
As $n \to\infty$, for the topology of uniform convergence on $[0,1]$,
%
\begin{equation}
\label{eqLm-limit} n^{-\beta/2}C_n \to\mathscr M\qquad\mbox{almost
surely and in }L^m\qquad \mbox{for all }m\in\mathbb{N}.
\end{equation}
Up to a multiplicative constant, the process $\mathscr M$ is the unique
solution of
(\ref{eqfixchord}) (in distribution) with c{\`a}dl{\`a}g paths
subject to
$\mathbf{E} [\|\mathscr M\|^2] < \infty$.
\end{theorem}

\begin{rem*}
The theorem states $\| n^{-\beta/2}C_n -
\mathscr M
\| \to0$
almost surely and in $L^m$ as $n \to\infty$. However, for technical
reasons of
measurability, the state space of c{\`a}dl{\`a}g functions $\mathcal
{D}[0,1]$ is
endowed with the
Skorokhod topology. We refer to the standard textbook by
Billingsley \cite{bil68} for refined information on this matter.
\end{rem*}

%
\begin{theorem}\label{thmdualtree}
Almost surely as $n\to\infty$, we have
\[
\bigl(\bigl(T_n, n^{-\beta/2} d_n\bigr), \mathfrak
L_n \bigr) \to\bigl((\mathcal T_{\mathscr
M}, d_{\mathscr M}),
\mathfrak L _{\mathscr M} \bigr).
\]
Here, the convergence of the components is with respect to the Gromov--Hausdorff
distance between compact metric spaces and the Hausdorff metric on
compact subsets of
the disk $\mathscr D$.
\end{theorem}

The assertion about the convergence of the lamination is the main
result in
\cite{legalcu} and we partially rely on their results to give a
simplified proof using
our approach.
The convergence of the tree does not use any statement in
\cite{legalcu}, and proves the conjecture in Section~4.4 of \cite{legalcu}.
Note that the number of chords $N_n$ inserted by time $n$ is of order
$\sqrt n$.
More precisely, Curien and Le~Gall \cite{legalcu} show that $N_n /
\sqrt {n} \rightarrow\sqrt{\pi}$ almost
surely. So the \emph{volume} of the tree $T_n$ is $N_n\sim(\pi
n)^{1/2}$ and the
order of magnitude of distances with respect to its volume is
$N_n^\beta$.

\begin{rem*}
In fact, it is not hard to see that $N_n$ is
distributed like the number of maxima in a triangle
when $n$ points are inserted uniformly at random and independent of
each other.
This quantity has been studied in detail by
Bai et al. \cite{baihwang}, who give exact formulas for the mean and
the second moment together with first order expansions of all higher
moments which imply asymptotic normality of $N_n$ after proper
rescaling. We refer to Theorem~3 in \cite{baihwang} for details.
\end{rem*}


The limit metric space $\mathcal T_\mathscr M$ is yet another natural
random fractal
real tree
which
does not come from a Brownian excursion \cite
{Aldous1991a,Aldous1991b,Aldous1993a} or
another more general L{\'e}vy process \cite{LeLe1998,DuLe2002}.
Other examples include the fragmentation trees of Haas and Miermont
\cite{HaMi2004a} (see also \cite{HaMi12}) and the minimum
spanning tree of a complete graph whose scaling
limit has been constructed by Addario-Berry et al. \cite{AdBrGoMi2012a}.

A priori, the process $\mathscr M$ is not fully identified by the
fixed-point equation
(\ref{eqfixchord}) because of the free multiplicative constant.
(Curien and Le~Gall \cite{legalcu} proved that the scaling constant
$\kappa$ in (\ref{eqfree-const}) exists. They did not need to identify
it for the
main topic there is the limit lamination, which is not affected by this leading
constant.) In order to identify $\mathscr M$ precisely, we study the
asymptotics of
$\mathbf{E} [C_n(\xi)]$ for an independent uniform random variable
$\xi$.
Let $\Gamma(s):=\int_0^\infty x^{s-1} e^{-x} \,dx$ denote the Gamma function.

%
\begin{theorem}\label{thmunif}
Let $\gamma= \beta/2 +1$ and $\bar\gamma= \frac{-\sqrt{17}
+1}{4}$. Then
%
\begin{eqnarray}
\label{eqmun} \qquad\mathbf{E} \bigl[C_n(\xi) \bigr] &=&\frac{\sqrt{\pi}}{4}
\sum_{k=1}^n \pmatrix{n \cr k}
(-1)^{k+1} \frac{\Gamma(k-\gamma+1) \Gamma(k-\bar\gamma+1)}{k! \Gamma
(k+3/2)\Gamma(2-\gamma)\Gamma(2- \bar\gamma)}.
\end{eqnarray}
Furthermore, as
$n\to\infty$,
%
\begin{eqnarray}\label{constunif}
\mathbf{E} \bigl[C_n(\xi)\bigr] &=& c n^{\beta/2}
+O(1)
\nonumber\\[-8pt]\\[-8pt]
\eqntext{\displaystyle\mbox{with } c = \frac {\sqrt{\pi}\Gamma(2 \gamma-1/2 )}{2 \Gamma(\gamma) \Gamma
^2(\gamma
+1/2)} = 1.178226\ldots.}
\end{eqnarray}
\end{theorem}

The asymptotic expansion in (\ref{constunif}) may be obtained from the
work of
Bertoin and Gnedin \cite{BeGn2004} on nonconservative fragmentations.
More precisely, the first-order
asymptotics is explicitly stated there, and the error term (that we
need to prove the
almost sure convergence in Theorems~\ref{thmmain} and~\ref{thmdualtree}) follows
from the same representation with a little more work. We include the explicit
formula (\ref{eqmun}) since it seems that similar developments have
attracted some
interest in the community of analysis of algorithms \cite{ChHw2003}.
Theorem~\ref{thmunif} is not the heart of the matter here, but for the
sake of
completeness, we provide a proof in \hyperref[secunif]{Appendix}.
Let $\mathrm{B}(x,y) = \int_0^1 t^{x-1} (1-t)^{y-1} \,dt$ denote the
beta function.

%
\begin{cor}The process $\mathscr M$ in (\ref{eqLm-limit}) is such that
\[
\mathbf{E} \bigl[\mathscr M(s)\bigr] = \kappa\bigl(s(1-s)\bigr)^\beta,
\qquad\kappa = \frac{c}{\mathrm{B}(\beta+ 1, \beta+1)} = 3.34443\ldots,
\]
where $c$ is given in (\ref{constunif}), which identifies uniquely
the solution
of (\ref{eqfixchord}) among all processes with c{\`a}dl{\`a}g paths subject
to $\mathbf{E} [\|\mathscr M\|^2] < \infty$.
\end{cor}

\subsubsection*{The homogeneous lamination}
The lamination process we have introduced is actually an instance of a
more general
fragmentation process which is also discussed in \cite{legalcu}, Section~2.4,
using a two-stage split
procedure: first pick a fragment with probability proportional to its
mass to the power
$\alpha$ (here $\alpha=2$), then choose the random chord within this
fragment by
sampling two independent uniform points on the intersection of the corresponding
fragment with the circle. In the
language of fragmentation theory \cite{Bertoin2006}, $\alpha$ is the
\emph{index of self-similarity}, and the actual split given the
fragment is described
by a \emph{dislocation measure}, which is here (essentially) given by
the two uniform
points conditioned to fall in the same fragment. One may define related
fragmentations
where the next fragment to split is chosen with probability
proportional to its mass
to the power $\alpha\in\ensuremath{\mathbb{R}}$; the cases of
interest here are those with
$\alpha\ge0$.
When $\alpha\ge0$, Curien and Le~Gall \cite{legalcu}, Section~2, have
shown that the limit laminations are all
identical.
However, and although it encodes the same lamination for
every $\alpha\ge0$, the encoding process (related to the dual tree)
depends on
whether $\alpha>0$ or $\alpha=0$. The tree $(\mathcal T_\mathscr M,
d_{\mathscr M})$ is the
scaling limit
of the dual tree for every $\alpha>0$, but this raises the question of
the dual tree in
the case $\alpha=0$.

When $\alpha=0$, the choice of the next fragment is independent of its mass---hence homogeneous---and there is a drastic change in the behavior
of the height process. At each step, the
fragment containing the next chord is chosen uniformly at random. Note
here that
every trial yields a new insertion, and the lamination at time~$n$
contains $n$
chords. Write $C_n^h(s)$ for the height in the dual tree of the fragment
containing $s\in[0,1]$. Curien and Le~Gall \cite{legalcu},
Theorem~3.13, prove that for every $s\in(0,1)$ the
quantity $n^{-1/3}C_n^h (s)$ converges almost surely as $n\to\infty$, where
the pointwise limit $\mathscr H(s)$ may be described by a process
$\mathscr H$ with continuous
sample paths which satisfy another, similar but different, fixed-point equation
(see~Section~\ref{seclimith}).

In this case, the approach used to prove Theorem~\ref{thmdualtree}
yields the
following result: let $T_n^{h}$ denote the tree dual to the homogeneous
laminations
$\mathfrak L_n^{h}$, and let $d_n^{h}$ denote the graph distance in $T_n^{h}$.

\begin{theorem}\label{thmdualtree0}
Almost surely, as $n\to\infty$, we have
\[
\bigl(\bigl(T_n^{h}, n^{-1/3}
d_n^{h}\bigr); \mathfrak L_n^{h}
\bigr) \to \bigl((\mathcal T_{\mathscr H},d_{\mathscr H}); \mathfrak
L_{\mathscr H}\bigr).
\]
The convergence of the dual tree is with respect to Gromov--Hausdorff topology,
and the lamination converges for the Hausdorff topology on compact
subsets of the disk.
\end{theorem}

The second assertion of the Theorem~\ref{thmdualtree0} has been
proved in
\cite{legalcu}, and our contribution relies in the proof of convergence
of the dual tree. Our approach to Theorem~\ref{thmdualtree0} relies on
the same
functional ideas developed for the self-similar case and the proof of
Theorem~\ref{thmdualtree}. We explain them below in more detail.

%
\begin{rem*}
As already indicated, we have $\mathfrak L_{\mathscr M} = \mathfrak
L_{\mathscr H}$ almost surely.
In terms of the dual trees, this corresponds to the fact that
the equivalence relations given by $d_{\mathscr M}$ and $d_{\mathscr
H}$ almost
surely identify the
same points on the unit interval.
This highlights the difference between convergence in the
Hausdorff distance solely relying on $\mathcal T_{\mathscr M} =
\mathcal T_{\mathscr H}$ as
collection of equivalence classes of $[0,1]$, not involving the
geometry of the
limit objects and convergence
of the dual trees to $(\mathcal T_\mathscr M,d_\mathscr M)$ [and
$(\mathcal T_\mathscr H,d_\mathscr H)$, resp.]
seen as metric spaces.
\end{rem*}

\subsubsection*{About the main ideas}
The main techniques in \cite{legalcu} are inherently pointwise, and one
of the main
difference in spirit\vadjust{\goodbreak} in our approach is to consider the problem as
functional from the
very beginning. In particular, we develop a new construction for the
limit process $\mathscr M$.
We construct the random process $Z$ (which is almost surely equal to
$\mathscr M$)
as the uniform limit of continuous functions $Z_n\dvtx [0,1]\to[0,\infty)$
which are
designed so that $(Z_n(s), n\ge0)$
is a nonnegative martingale for every $s\in[0,1]$.
Unlike in \cite{legalcu} where results entirely rely on an approach
that is \emph{forward} in time, we make use of the inherent recursive
structure of
the problem and study the telescoping sum representation
\[
Z_n-Z_0=\sum_{i=0}^{n-1}
(Z_{i+1}-Z_i).
\]
More precisely, this \emph{backward} approach is based on an $L^2$
argument using the fact that one can bound
$\|Z_{i+1}-Z_i\|^2$ in terms of \emph{independent} copies of
$\|Z_i-Z_{i-1}\|$ corresponding to the two fragments created by the
insertion of the first chord as in (\ref{eqfixchord}).
The expansion of the square yields one contribution
involving the single fragment ($\mathbf{E} [\|Z_{i}-Z_{i}\|^2]$) and one
involving the
first two fragments, which may be bounded using only $\mathbf{E}
[|Z_{i}(\xi )-Z_{i-1}(\xi)|^2]$
for a uniform random variable $\xi$. So our representation allows to
leverage the convergence at a uniformly distributed random point
(Lemma~\ref{lemonedim}) to deduce
$\mathbf{E} [\|Z_{i+1}-Z_i\|^2]\le\chi\cdot\mathbf{E} [\|
Z_{i}-Z_{i-1}\|^2]$ for
some $\chi<1$ and all $i$ sufficiently large
leading to geometric convergence in a functional sense.
The convergence of the discrete sequence $n^{-\beta/2}C_n$ is obtained
in a similar vein. After using an appropriate embedding of both
the sequence and the limit $Z$, our backward approach technically relies
on ideas from the contraction method
\cite{Ro91,RaRu95,NeRu04}; see also \cite{NeSu12} for a recent
development in
function spaces. A~somewhat similar approach toward functional
convergence results relying on first establishing one-dimensional
convergence at a specific point in the context of
the Quicksort algorithm can be found in Ragab and R{\"o}sler \cite{ragabroe}.
%

\subsection{Related work on random laminations of the disk}

The work of Curien and Le~Gall \cite{legalcu} was motivated by the
pioneering work of
Aldous \cite{al94circle,al94} who studied \emph{uniform random
triangulations} of the disk
which arise as limiting objects for uniform triangulations of regular
$n$-gons as
$n \to\infty$. In the case of uniform random triangulations, the
process which
encodes the limit triangulation is the Brownian excursion, and the
scaling limit
of the sequence of dual trees is the
Brownian continuum random tree introduced in \cite
{Aldous1991a,Aldous1991b,Aldous1993a}.

Among the recent work on laminations of the disk, one can mention \cite
{CuKo2012a}
where Curien and Kortchemski showed that the Brownian triangulation is
also the
scaling limit of other random subsets of the disk, in particular,
noncrossing trees
(sets of noncrossing chords which form a tree) \cite{MaPa2002}, and dissections
(noncrossing sets of chords) under the uniform distribution. By sampling
tessellations according to a Boltzmann weight depending on the degree
of the faces,
Kortchemski \cite{Ko2011b} obtained limit laminations which are not
triangulations and are
encoded by excursions of stable spectrally positive L\'evy processes
(with L\'evy measure concentrated on $[0,\infty)$).
Finally, Curien and Werner \cite{CuWe2011a} have
studied geodesic laminations of the Poincar\'e disk.
They construct and study the unique random
tiling of the hyperbolic plane into triangles with vertices on the
boundary whose
distribution is invariant under M\"obius transformations and satisfies
a certain
spatial Markov property.

\subsubsection*{Plan of the paper}
In Section~\ref{seclimit}, we give our
construction of a
continuous solution $Z$ of (\ref{eqfixchord}) with $\mathbf{E}
[Z(s)] =
\kappa
(s(1-s))^\beta$.
(Recall that $Z=\mathscr M$ almost surely.)
The construction guarantees finiteness of all moments of the supremum
$\|Z\|$ which is
essential for our approach. Here, we also prove the characterization of
$Z$ as a
solution of (\ref{eqfixchord}) under additional conditions. In
Section~\ref{secconv},
we prove the uniform convergence of $n^{-\beta/2} C_n$ to $Z$.
We also obtain an upper bound on the rate of convergence in the $L^m$ distance,
$m\ge1$, which yields the almost sure convergence in Theorem~\ref{thmmain}.
Here, we also show how our results simplify the arguments to deduce
convergence of
the lamination. Section~\ref{seclimith} is devoted to the proof of
Theorem~\ref{thmmain2} which covers the homogeneous case $\alpha= 0$.
Finally, in Section~\ref{secpropdual} we prove some properties about
the dual
tree $\mathcal T_Z=\mathcal T_\mathscr M$, in particular about its
fractal dimension.
Our proof of Theorem~\ref{thmunif} is based on generating functions
and is given in
\hyperref[secunif]{Appendix} to keep the body of the paper more focused.


\section{A functional construction of the self-similar limit height
process} \label{seclimit}

Our aim in this section is to propose an alternative construction of
the limit process~$Z$. Although Curien and Le~Gall \cite{legalcu}
have proved the existence of a continuous process $\mathscr M$ (which
is almost
surely equal to $Z$) using bounds on the moments on the increments and
Kolmogorov's
criterion \cite{ReYo1991}, Theorem~2.1, we adopt here a
functional approach
that will later guide our proof of the convergence theorem
(Theorem~\ref{thmmain}).

The process $Z$ is constructed in terms of a set of independent random variables
on the unit interval as follows.
First, we identify the nodes of the infinite binary tree with the set
of finite
words on the alphabet $\{0,1\}$,
\[
\mathcal T= \bigcup_{n\ge0}\{0,1\}^n.
\]
The descendants of $u\in\mathcal T$ correspond to all the words in
$\mathcal T$ with
prefix $u$.
Let $\{(U_v,V_v), v\in\mathcal T\}$
be a set of independent and identically distributed two-dimensional
random vectors with density $2 \mathbf{1}_{ \{ 0 \leq u \leq v \leq1
\} }$ and $A^+ = \{ (u,v) \in[0,1]^2\dvtx  u < v\}$. For convenience,
we also set
$U:= U_\varnothing$ and $V:= V_\varnothing$ and define
\[
h(s) = \bigl(s(1-s)\bigr)^\beta.
\]

Let $\mathcal C_0([0,1])$ denote the set of continuous functions $f$ on
the unit
interval vanishing at the boundary, that is, $f(0)=f(1)=0$.
Define the operator $G\dvtx  A^+ \times\mathcal C_0([0,1])^2 \to\mathcal
C_0([0,1])$ by
\[
G[u,v;f_0,f_1](s)
= \cases{ \displaystyle\bigl(1-(v-u)
\bigr)^\beta f_0 \biggl( \frac{s}{1-(v-u)} \biggr),\qquad \mbox{if $s<U$},
\vspace*{5pt}\cr
\displaystyle\bigl(1-(v-u)\bigr)^\beta
f_0 \biggl( \frac{u}{1-(v-u)} \biggr) + (v-u)^\beta
f_1 \biggl( \frac{s-u}{v-u} \biggr),
\vspace*{3pt}\cr
\hspace*{170pt}\mbox{if $U\le s<V$,}
\vspace*{5pt}\cr
\displaystyle\bigl(1-(v-u)\bigr)^\beta f_0 \biggl(
\frac{s- (v-u)}{1-(v-u)} \biggr),\qquad\mbox{if $s\ge V$.}}
\]
For convenience, define
%
\begin{eqnarray}\label{defK0K1}
K_0(s;u,v) & = &\mathbf{1}_{ \{ s<u \} } \frac{s}{1-(v-u)}+
\mathbf {1}_{ \{ s \geq v \} } \frac{s-(v-u)}{1-(v-u)}\nonumber
\\
&&{}  + \mathbf{1}_{ \{ u \leq s < v \} }
\frac{u}{1-(v-u)},
\\
K_1(s;u,v) & = &\mathbf{1}_{ \{ u \leq s < v \} }
\frac{s-u}{v-u},\nonumber
\end{eqnarray}
so that we have the more compact form
\[
G[u,v;f_1,f_2](s)= \bigl(1-(v-u)\bigr)^\beta
f_0\bigl(K_0(s;u,v)\bigr)+ (v-u)^\beta
f_1\bigl(K_1(s;u,v)\bigr).
\]

For every node $u\in\mathcal T$, let $Z_0^{(u)}= \kappa h(s)$. Then define
recursively
%
\begin{equation}
\label{defrecZ} Z_{n+1}^{(u)}=G\bigl(U_u,
V_u; Z_n^{(u0)}, Z_n^{(u1)}
\bigr),
\end{equation}
%
and define $Z_n=Z_n^\varnothing$ to be the value observed at the root
of $\mathcal T$.
For every $s\in(0,1)$, one can verify that the sequence $(Z_n(s), n\ge
0)$ is a
nonnegative martingale for the filtration $\mathcal F_n=\sigma((U_u,V_u)\dvtx
|u|\le n)$,
so that $Z_n(s)$ converges almost surely. (This reduces to proving that
$\mathbf{E} [G(U,V;h,h)(s)]=h(s)$ for every $s\in[0,1]$, and is essentially
proved in
the first moment calculation in Section~4.2 of \cite{legalcu}.)
The game is now to prove that this
convergence is actually uniform for $s\in(0,1)$, which will yield the following
theorem. (See Figure~\ref{figlamination-process} for a simulation.)

%
\begin{figure}

\includegraphics{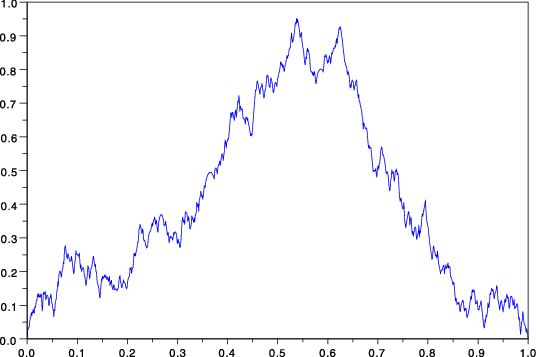}

\caption{An instance of limit height process $Z/\kappa$ simulated
using the approximations $Z_n/\kappa, n\ge1$.}\label{figlamination-process}
\end{figure}

%
\begin{theorem} \label{thmfixas}
For any $u \in\mathcal T$, almost surely, the sequence $Z_n^{(u)}$
converges uniformly to a
continuous process $Z^{(u)}$.
Almost surely, for every $s\in[0,1]$, we have
%
\begin{equation}\label{fixpas}
Z(s)=\cases{ \displaystyle\bigl(1-(V-U)\bigr)^\beta
Z^{(0)} \biggl( \frac{s}{1-(V-U)} \biggr),\qquad\mbox{if $s<U$,}
\vspace*{5pt}\cr
\displaystyle\bigl(1-(V-U)\bigr)^\beta Z^{(0)} \biggl(
\frac{U}{1-(V-U)} \biggr) + (V-U)^\beta Z^{(1)} \biggl(
\frac{s-U}{V-U} \biggr),
\vspace*{3pt}\cr
\hspace*{190pt}\mbox{if $U\le s<V$,}
\vspace*{5pt}\cr
\displaystyle\bigl(1-(V-U)
\bigr)^\beta Z^{(0)} \biggl( \frac{s- (V-U)}{1-(V-U)} \biggr),\qquad\mbox{if $s\ge V$.}}\hspace*{-30pt}
\end{equation}
Moreover, $\mathbf{E} [\|Z\|^m] < \infty$ for all $m \in\mathbb
{N}$, $\mathbf{E}  [Z(s)  ] =
\kappa h(s)$ and writing $\mathcal{L}(X)$ for the law of $X$, we have
$\mathcal{L}
(Z^{(u)}) = \mathcal{L}(Z)$ for all $u \in\mathcal T$. 
\end{theorem}

From the theorem above and its proof given subsequently, one deduces
that the random variable $Y=Z(\xi)$ where $\xi$ is an independent
uniformly distributed random variable satisfies the following
distribution fixed-point equation:
%
\begin{equation}
\label{eqfixunif} Y \stackrel{d} {=} \bigl(1-(V-U)\bigr)^\beta Y +
\mathbf{1}_{ \{ U \leq\xi< V
\} } (V-U)^\beta \widehat Y.
\end{equation}
Here, $\widehat Y$ is distributed as $Y$ and the random variables
$Y$, $\widehat Y$, $\xi$, and $(U,V)$ are independent.
In fact, this identity is at the very heart of the proof of Theorem
\ref{thmfixas}
as will become clear below. 

Note that any set of independent vectors $(U_v, V_v)_{v \in\mathcal
T}$ with distribution $2 \mathbf{1}_{ \{ 0 < u < v <1 \} }$ can be
used in the
construction given in this section. It is in the
next section where we make a specific choice of the set in order to
couple the limit to the discrete lamination process. In \cite
{sodawir,BrNeSu12}, a similar construction has been used in a related
context: there, the uniform convergence follows from a bootstrap of the
pointwise convergence which requires tedious verifications. Here, we
prove directly that the convergence is uniform using an $L^2$ argument.

Write $\Psi= U / (1-(V-U))$. By the definition of $Z_n$, we have the
following expansion:
%
\begin{eqnarray}\label{recstart}
\qquad && \bigl[Z_{n+1}(s) - Z_n(s)
\bigr]^2\nonumber
\\
&&\qquad =  \bigl(1-(V-U)\bigr)^{2\beta} \bigl[Z_n^{(0)}
\bigl(K_0(s;U,V)\bigr) - Z_{n-1}^{(0)}
\bigl(K_0(s;U,V)\bigr) \bigr]^2\nonumber
\\
&&\quad\qquad{} + \mathbf{1}_{ \{ U \leq s < V \} } (V-U)^{2\beta} \bigl[Z_n^{(1)}
\bigl(K_1(s;U,V)\bigr) - Z_{n-1}^{(1)}
\bigl(K_1(s;U,V)\bigr) \bigr]^2
\\
&&\quad\qquad{} + \mathbf{1}_{ \{ U \leq s < V \} } 2\bigl((V-U) \bigl(1-(V-U)\bigr)
\bigr)^\beta \bigl[Z_n^{(0)}(\Psi) -
Z_{n-1}^{(0)}(\Psi) \bigr] \nonumber
\\
&&\qquad\qquad{} \times\bigl[Z_n^{(1)}\bigl(K_1(s;U,V)\bigr) -
Z_{n-1}^{(1)}\bigl(K_1(s;U,V)\bigr) \bigr].\nonumber
\end{eqnarray}
Define
%
\begin{eqnarray}\label{constlip}
q &=& \mathbf{E} \bigl[\bigl(1-(V-U)\bigr)^{2\beta}\bigr] + \mathbf{E}
\bigl[(V-U)^{2\beta}\bigr]  \nonumber
\\
&=& \frac{2}{2\beta +1} = \frac{2}{\sqrt{17}-2} < 1,
\\
q' & = & 2 \sqrt{\mathbf{E} \bigl[
\bigl((V-U) \bigl(1-(V-U)\bigr)\bigr)^{2\beta}\bigr]}.
\nonumber
\end{eqnarray}
Then equation (\ref{recstart}) yields
%
\begin{eqnarray}\label{recmeansq}
&& \mathbf{E} \bigl[\|Z_{n+1} - Z_n
\|^2 \bigr]\nonumber
\\
&&\qquad \leq q \mathbf{E} \bigl[\|Z_{n} -
Z_{n-1}\|^2 \bigr]\nonumber
\\
&&\quad\qquad{}+ 2\mathbf{E} \bigl[\bigl((V-U) \bigl(1-(V-U)\bigr)\bigr)^\beta
\nonumber\\[-8pt]\\[-8pt]
&&\hspace*{60pt}{}\times
\bigl|Z_n^{(0)}(\Psi) - Z_{n-1}^{(0)}(
\Psi)\bigr| \cdot\bigl\| Z_n^{(1)}- Z_{n-1}^{(1)}\bigr\|\bigr]
\nonumber
\\
&&\qquad \leq q \mathbf{E} \bigl[\|Z_{n} - Z_{n-1}\|^2
\bigr] \nonumber
\\
&&\quad\qquad{} + q' \sqrt {\mathbf{E} \bigl[\|Z_{n}-
Z_{n-1}\| ^2 \bigr] \cdot\mathbf{E} \bigl[
\bigl(Z_n^{(0)}(\Psi) - Z_{n-1}^{(0)}(\Psi)
\bigr)^2\bigr]}\nonumber
\end{eqnarray}
by the Cauchy--Schwarz inequality. 
If we were to drop the second term in the last line above, we would
have geometric convergence of $\mathbf{E} [\|Z_{n+1}-Z_{n}\|^2]$ since $q<1$.
Now, the crucial observation is that, the second term may actually be
shown to decrease geometrically using only the convergence at a
uniformly random point.
More precisely, the random variable $\Psi$ is uniform on $[0,1]$ and
independent of $\{(U_v, V_v)\dvtx  v \in\mathcal T_0\}$, where $\mathcal
T_0=\{0u\dvtx  u\in\mathcal T\}$. Thus,
\[
Z_n^{(0)}(\Psi) - Z_{n-1}^{(0)}(\Psi)
\stackrel{d} {=} Z_n(\xi)- Z_{n-1}(\xi),
\]
where $\xi$ is uniformly distributed on the unit interval and
independent of $\{(U_v, V_v)\dvtx v \in\mathcal T\}$. The following lemma
allows to bound the second summand in (\ref{recmeansq}).

\begin{lem} \label{lemonedim}
There exists a constant $C_1> 0$ such that, for all natural number
$n\ge1$,
\[
\mathbf{E} \bigl[\bigl(Z_n(\xi)- Z_{n-1}(\xi)
\bigr)^2\bigr] \leq C_1^2 q^n.
\]
\end{lem}

For the sake of clarity, we take Lemma~\ref{lemonedim} for granted for
now and show that it indeed implies exponential bounds for $\mathbf{E}
[\| Z_{n+1}-Z_n\|^2]$.

\begin{lem}\label{lemsqsup}For any $0 < \eta< 1-q^{1/2}$, there
exists a constant $C_2$ such that for all $n\ge0$,
\[
\mathbf{E} \bigl[\|Z_{n+1}- Z_n\|^2\bigr]\le
C_2 \bigl(q^{1/2}+\eta\bigr)^n.
\]
\end{lem}

\begin{pf}
Write $\Delta_n = \mathbf{E} [\|Z_{n+1}- Z_n\|^2]$ for $n \geq0$.
Then the
inequalities (\ref{recmeansq}) and Lemma~\ref{lemonedim} yield, for
every natural number $n$,
%
\begin{equation}
\label{recdelta} \Delta_n \leq q \Delta_{n-1} +
C_1 q' q^{n/2}\Delta_{n-1}^{1/2}.
\end{equation}
%
First, (\ref{recdelta}) clearly implies that $\Delta_n$ is bounded:
we have
\[
\Delta_n \le\bigl(q+ C_1q'
q^{n/2}\bigr) \cdot\max\{\Delta_i\vee1\dvtx  0\le i<n\}.
\]
So, taking $n_0$ large enough that $q+C_1q' q^{n/2}<1$ for all $n\ge
n_0$, it follows that for all $n\ge n_0$, we have $\Delta_n \le\max\{
\Delta_i, i\le n_0\} \vee1$.

Now, fix $0 < \eta< 1-q^{1/2}$ and $M$ such that $\Delta_n \leq M^2$
for all $n \in\mathbb{N}_0$.
Let $n_1$ be large enough such that for any $n \geq n_1$, we have
\[
\frac{M C_1 q'}{q^{1/2}+ \eta} \biggl( \frac{q^{1/2}}{q^{1/2}+ \eta} \biggr)^n \leq1.
\]
We now proceed by induction on $n\ge n_1$. Assume that $\Delta_n \leq
C_2 (q^{1/2} + \eta)^n$ for \mbox{$n \leq n_1$} where the constant $C_2$ is
chosen large enough such that
$q C_2 \leq( C_2 - 1)(q^{1/2} + \eta)$. Then by (\ref{recdelta}), we have
\begin{eqnarray*}
\Delta_{n+1} & \leq& q C_2 \bigl(q^{1/2} + \eta
\bigr)^n + M C_1 q' q^{n/2}
\\
& \leq&\bigl(q^{1/2} + \eta\bigr)^{n+1} \biggl[
\frac{C_2 q}{q^{1/2} + \eta} + \frac{M C_1 q'}{q^{1/2}+ \eta} \biggl( \frac{q^{1/2}}{q^{1/2} + \eta}
\biggr)^n \biggr]
\\
& \leq& C_2 \bigl(q^{1/2} + \eta\bigr)^{n+1}
\end{eqnarray*}
by our choice for $C_2$ and $n_1$, which completes the proof.
\end{pf}

With Lemma~\ref{lemsqsup} in hand, we may now complete the proof of
Theorem~\ref{thmfixas}.
\begin{pf*}{Proof of Theorem~\ref{thmfixas}}
The fact that there exists a continuous process $Z$ such that $Z_n \to
Z$ uniformly almost surely follows from standard arguments: first,
Markov's inequality, monotone convergence and the Cauchy--Schwarz
inequality imply that $\sup_{m \geq n}\|Z_m - Z_n\| \to0$ in
probability. It easily follows that also
$\sup_{m,p \geq n}\|Z_m - Z_p\| \to0$ in probability. By monotonicity,
the latter convergence is actually almost sure. Thus, the sequence
$(Z_n, n\ge0)$ is almost surely uniformly Cauchy. Since $Z_n$ is
continuous for every $n\ge0$, the completeness of $\mathcal{C}[0,1]$
implies the
existence of a limit function $Z$ which is almost surely continuous.


The sequences $Z_n^{(0)}$ and $Z_n^{(1)}$, $n\ge1$, also converge
since they are both distributed like $Z_{n-1}$, $n\ge1$. Write
$Z^{(0)}$ and $Z^{(1)}$ their uniform almost sure limits. Letting $n
\rightarrow\infty$, in the definition of $Z_n$ in (\ref{defrecZ})
implies the equality in (\ref{fixpas}). \par
Next, we prove that 
$\sup_{n \geq1} \mathbf{E} [\| Z_n \|^m] < \infty$ for any $m \in
\mathbb{N}$ by
induction on $m$. It is true for $m=1,2$ and we assume it holds for any
$\ell< m$ with $m > 2$. Then, by construction
%
\begin{eqnarray}\label{ineqhighermoments}
\mathbf{E} \bigl[\|Z_n\|^m\bigr] &\leq& \underbrace{
\bigl(\mathbf{E} \bigl[\bigl(1-(V-U)\bigr)^{m\beta}\bigr] + \mathbf{E}
\bigl[(V-U)^{m\beta}\bigr] \bigr) }_{q_m} \cdot\mathbf{E} \bigl[\|
Z_{n-1}\|^{m}\bigr]
\nonumber\\[-8pt]\\[-8pt]
&&{} + \sum_{i=1}^{m-1} \pmatrix{m \cr i}
\mathbf{E} \bigl[\|Z_{n-1}\|^i\bigr] \mathbf {E} \bigl[
\|Z_{n-1}\| ^{m-i}\bigr]. \nonumber
\end{eqnarray}
The induction hypothesis implies that the summand in (\ref{ineqhighermoments}) is bounded uniformly in $n$. Since $q_m< 1$, an
easy induction on $n$ gives the desired result $\sup_{n \geq1}
\mathbf{E}  [\| Z_n\|^m  ] < \infty$. It follows that
$\mathbf{E} [\|Z\|^m] <\infty$ for any $m
\in\mathbb{N}$.

Finally,
the fact that $\mathbf{E}  [Z_n(s)  ]= \kappa h(s)$ for all
$n$, and thus $\mathbf{E} [Z(s)] =
\kappa h(s)$, is essentially equivalent to the martingale property of $Z_n(s)$
mentioned above and can be found in \cite{legalcu}, Section~4.2.
This completes the proof.
\end{pf*}

It now remains to prove Lemma~\ref{lemonedim} about the bound on
$\mathbf{E} [|Z_n(\xi)-Z_{n-1}(\xi)|^2]$.
\begin{pf*}{Proof of Lemma~\ref{lemonedim}}
%
Let $W_0 = K_0(\xi, U,V)$ and $W_1 = K_1(\xi, U,V)$.
The key ingredient of the proof is the following observation:
\begin{enumerate}[(O1)]
\item[(O1)] On the event $\{ \xi\notin(U,V] \}$, the quantities
$W_0$ and $V-U$ are independent and $W_0$ has uniform distribution.
Moreover, given the event $\{U \leq\xi< V\}$, the quantities
$W_0$, $W_1$ and $V-U$ are independent and both $W_0$ and $W_1$ have
uniform distribution.
\end{enumerate}
%
Using this observation in (\ref{recstart}) directly implies the
desired result
\[
\mathbf{E} \bigl[\bigl|Z_{n+1}(\xi) - Z_n(\xi)\bigr|^2
\bigr] \leq q \mathbf{E} \bigl[\bigl|Z_n(\xi) - Z_{n-1}(
\xi)\bigr|^2\bigr],
\]
since the mixed terms cancel out for the expected value $\mathbf{E}
[Z_n^{(u)}(\xi)]=\kappa\mathrm{B}(\beta+ 1, \beta+ 1)$ is independent
of $u\in\mathcal T$ and $n\ge0$.
\end{pf*}
To complete this section, we show that the process $Z$ we have
constructed is characterized by the fixed-point equation (\ref{fixpas})
in a reasonable class of processes.

\begin{prop} \label{propunique}
The process $Z$ is the unique solution of the fixed-point equation
(\ref{eqfixchord}) (in distribution) among all c{\`a}dl{\`a}g
processes subject to $\mathbf{E} [Z(\xi)] = \kappa\mathrm{B} (\beta
+1, \beta
+1)$ and $\mathbf{E} [\|Z\|^2]< \infty$.
\end{prop}

\begin{pf}
The main part of the proof relies on the functional contraction method
developed in \cite{NeSu12}.
Let $\mathcal{M}(\mathcal{D}[0,1])$ denote the set of probability
measures on $\mathcal{D}[0,1]$.
Consider the map $T\dvtx  \mathcal{M}(\mathcal{D}[0,1]) \to\mathcal
{M}(\mathcal{D}[0,1])$ which to
$\mu\in\mathcal M(\mathcal{D}[0,1])$ assigns the law of the process
\[
\bigl(1-(V-U)^\beta\bigr) X \bigl(K_0(s; U,V)\bigr) +
(V-U)^\beta\widehat X\bigl(K_1(s;U,V)\bigr),
\]
where $X, \widehat X$ are independent functions sampled according to
$\mu$, both independent of $(U,V)$. Let $\mathcal{M}_2(h) \subseteq
\mathcal{M}(\mathcal{D}[0,1])$
be the subset of measures $\mu$ such that if $X$ is $\mu$-distributed
then $\mathbf{E} [\|X\|^2] < \infty$ and $\mathbf{E} [X(t)] = h(t)$
for all $t\in[0,1]$.
Lemma~18 in \cite{NeSu12} asserts that $T$ is contractive with respect
to the Zolotarev metric $\zeta_2$ in the space $\mathcal{M}_2(h)$ where
the Lipschitz constant can be chosen as $q<1$ given in (\ref{constlip}). This lemma relies on a discrete sequence [denoted $(X_n)$
there] satisfying conditions (C1), (C2), (C3)
formulated on page~20 in \cite{NeSu12}. In our setting, as we deal
directly with the limiting fixed-point equation, conditions (C1)~and~(C3) reduce to the fact that $T(\mathcal{M}_2(h))
\subseteq
\mathcal{M}_2(h)$. This can be shown by a direct computation (see,
e.g., \cite{legalcu}, Section~4.2).
Condition (C2) is satisfied for we have $q = \mathbf{E}
[(1-(V-U))^{2\beta}] + \mathbf{E} [(V-U)^{2\beta}] < 1$.
Furthermore, Curien and Le~Gall~\cite{legalcu}, Section~4.2, prove that
the mean function of any solution $X$ of (\ref{eqfixchord}) with
$\int_0^1 \mathbf{E} [|X(s)|] < \infty$ is a multiple of $h$. This
completes the proof.
\end{pf}
\section{Convergence of the discrete process} \label{secconv}

\subsection{Notation and setting} \label{subsecnotset}

Let $(U_i', V_i')_{i \geq1} $ be a sequence of independent vectors,
where for each $i \geq1$, $U_i', V_i'$ are independent and have
uniform distribution on the unit interval. We consider the lamination
process built from this set of vectors as explained in the
\hyperref[secintro]{Introduction}. Let us first explain the connection with the tree-based
construction of Section~\ref{seclimit}. 
It should be intuitively clear how the family $(U_v, V_v)_{v \in
\mathcal T}$ used to build the limit is constructed from
$(U_i',V_i')_{i \geq1}$; the precise statement requires additional notation.

Initially, there is a single fragment $S^{\varnothing}$ consisting of
the entire
disk $\mathscr D$, which is associated to the root $\varnothing\in
\mathcal T$.
For $n=1$, $U_1'$ and $V_1'$ of course both fall inside the unique
fragment and we
insert the chord connecting $U_1'$ and $V_1'$. This chord divides
$S^{(\varnothing)}$
into two fragments $S^{(0)}$, $S^{(1)}$, where $S^{(0)}$ denotes the fragment
containing $0$.
Define
\[
U:= U_\varnothing= \min\bigl(U_1',V_1'
\bigr)\quad\mbox{and}\quad V:=V_\varnothing =\max\bigl(U_1',V_1'
\bigr).
\]
In general, at some stage $n-1$, $n > 1$, of the process, we have
inserted some chords,
and associated fragments to the nodes in a finite subtree $T_{n-1}$ of
$\mathcal T$
(a~connected set containing the root).
Then, at step time $n>1$, if there is no node $u\in T_n$ such that one
of $U_n', V_n'$
falls inside $S^{(u)}$ and the other outside (i.e., the chord
connecting $U_n'$ and
$V_n'$ does not intersect any other previously inserted chord), we
insert the chord
connecting them. Let $S^{(v)}$ be\vspace*{1.5pt} the smallest fragment containing both
$U_n'$ and
$V_n'$; the\vspace*{1.5pt} chord joining $U_n'$ to $V_n'$ splits $S^{(v)}$ into two fragments
$S^{(v0)}$ and $S^{(v1)}$; the labeling is chosen such that $v0$ is
closer to
the root in the dual tree.
Moreover, writing $\operatorname{Leb}$ for Lebesgue measure on the
circle $\mathscr C
$, we let
%
\begin{equation}
\ell^{(v)} = \operatorname{Leb}\bigl(S^{(v)}\cap\mathscr C
\bigr) \label{deflv}
\end{equation}
be the mass of the fragment $S^{(v)}$, and
\begin{eqnarray*}
U_v&=& \frac{ \min(U_n', V_n')- \operatorname{Leb} ( \{ s \notin
S^{(v)} \cap
\mathscr C\dvtx  0<s \leq\min(U_n', V_n') \})}{\ell^{(v)}},
\\
V_v&=& \frac{\max(U_n', V_n')- \operatorname{Leb}( \{ s \notin
S^{(v)} \cap\mathscr C\dvtx  0<s \leq\max(U_n', V_n') \})}{\ell^{(v)}}.
\end{eqnarray*}
Then $(U_v, V_v)_{v \in\mathcal T}$ is a set of independent random
vectors each having density $2 \mathbf{1}_{ \{ 0 < u < v < 1 \} }$. In
the following,
for any $u \in\mathcal T$, $Z^{(u)}$ will denote the process
constructed in Section~\ref{seclimit} using this set of vectors.


For any $n \in\mathbb{N}$, let $\tau_0(n)$ be the first time $k$
when there
exist exactly $n$ integers $2 \leq\ell_1 < \cdots< \ell_n = k$ such\vspace*{-1pt} that
$U_{\ell_i}', V_{\ell_i}'$, $i = 1, \ldots, n$, both take values in $S^{(0)}$
(actually $S^{(0)}\cap\mathscr C$).
Analogously, let $\tau_1(n)$ be defined in the same way using the
segment $S^{(1)}$. Observe that $\tau_0(n)$ and $\tau_1(n)$ are only
the stopping times when $n$ \emph{trials} have been made in $S^{(0)}$
and $S^{(1)}$, respectively, and that these trials may not have all led
to the successful insertion of a chord.
For $n \in\mathbb{N}$, let $C_n^{(0)}(s)$ be the number of chords intersecting
the straight line going from $0$ to $\varphi_0(s)$ defined by
%
\[
\varphi_0(s)= s\bigl(1-(V-U)\bigr) + (V-U) \mathbf{1}_{ \{ s > \Psi \} }
\]
at time $\tau_0(n)$. Here and for the remainder of this section,\vspace*{1pt} we use
$\Psi:=U/(1+U-V)$ as in Section~\ref{seclimit}.
Note that $\varphi_0(s)$ is the natural parameterization\vspace*{1pt} 
for $C_n^{(0)}$ in the sense that $(C_n^{(0)})_{n \geq0}$ has the same
distribution as $(C_n)_{n \geq0}$. Observe that $C_n^{(0)}$ is~well
defined since $\tau_0(n) < \infty$ almost surely for all $n \in
\mathbb{N}$.
Analogously, let $C_n^{(1)}(s)$ be the number of chords intersecting
the straight line going from $U$ to $\varphi_1(s)$ where
\[
\varphi_1(s) = s(V-U) + U
\]
at the stopping time $\tau_1(n)$. For convenience, let $C_0^{(0)} =
C_0^{(1)} \equiv0$. At time $n$, let $I^{(0)}_n$ and $I^{(1)}_n$ be
the number of
pairs among $(U_i', V_i')$, $i = 2,\ldots, n$, whose components both
fall in $S^{(0)}$ and $S^{(1)}$, respectively.
Finally, let $F_n = n-1-I^{(0)}_n - I^{(1)}_n$ be the number of
failures, or unsuccessful insertion attempts by time $n$ due to one
point falling in $S^{(0)}$ and the other in $S^{(1)}$.
Then, given the first chord $(U,V)$, 
%
\begin{eqnarray*}
&& \mathcal{L}\bigl(I^{(0)}_n, I^{(1)}_n,
F_n\bigr)
\\
&&\qquad = \operatorname{Multi}\bigl(n-1; \bigl(1-(V-U)
\bigr)^2, (V-U)^2, 2 (V-U) \bigl(1-(V-U)\bigr)\bigr).
\end{eqnarray*}
Almost surely, for every $s\in[0,1]$,
we have
%
\begin{equation}
\label{eqreccn} C_n(s)= \cases{ \displaystyle C^{(0)}_{I^{(0)}_n}
\biggl( \frac{s}{1-(V-U)} \biggr),\qquad\mbox{if $s<U$,}
\vspace*{5pt}\cr
\displaystyle
C^{(0)}_{I^{(0)}_n} \biggl(\frac{U} {1-(V-U)} \biggr) + 1+
C^{(1)}_{I^{(1)}_n} \biggl( \frac{s-U}{V-U} \biggr),
\vspace*{3pt}\cr
\hspace*{121.5pt} \mbox{if $U \le s<V$,}
\vspace*{5pt}\cr
\displaystyle C^{(0)}_{I^{(0)}_n} \biggl(
\frac{s-(V-U)}{1-(V-U)} \biggr),\qquad \mbox{if $s\ge V$.}}
\end{equation}
Let $\xi$ be a uniform random variable, independent of $(U_i', V_i')_{i
\geq1}$. Then, we let $X_n$ be the following rescaled version of
$C_n$, for any $n\ge1$:
\[
X_n(s):= C_n(s) \frac{\kappa\mathrm{B}(\beta+ 1, \beta
+1)}{\mathbf{E} [C_n(\xi)]}.
\]

\subsection{Uniform convergence in $L^2$}

The main result of this section is the following theorem.

\begin{theorem} \label{thmconvunif}
As $n \to\infty$, we have $\mathbf{E} [\|X_n - Z\|^2] \rightarrow0$.
\end{theorem}

The convergence in $L^2$ will be used in Section~\ref{secLm} as the
base case of an inductive argument showing that one actually has
uniform convergence in every $L^m$, $m\ge2$. 

The proof runs along similar lines as the construction of the limit
process. It resembles ideas from the area of the contraction method
such as in \cite{NeRu04,NeSu12}. However, note that we are working
with a coupling of the process to its limit; we do not introduce any
metrics on a space of probability measures. The proof relies on the
same trick which allowed us to construct the limit process $Z$ in
Section~\ref{seclimit}, namely a bootstrapping of the convergence at a
uniform point which is made possible by the immediate decoupling of the
processes in two fragments when a chord is added.

In the following, given a real valued random variable $Y$, we write $\|
Y\|_2$ for the $L^2$-norm of $Y$ defined by $\mathbf{E}
[|Y|^2]^{1/2}$. The
convergence at a uniformly random location reads:

\begin{lem}\label{lemconvunif}Let $\xi$ be a $[0,1]$-uniform random
variable independent of $(U_i', V_i')_{i \geq1}$. Then
%
\begin{equation}
\label{l2one} \lim_{n \to\infty} \bigl\llVert X_{n}(\xi)-
Z( \xi) \bigr\rrVert _2 = 0.
\end{equation}
\end{lem}

We postpone the proof and show immediately how one leverages this
information to prove that $X_n\to Z$ uniformly in $L^2$ as $n\to\infty$.

\begin{pf*}{Proof of Theorem~\ref{thmconvunif}}
Let $\mu(n) = \mathbf{E}  [C_n(\xi)  ]$, where $\xi$ is
an independent uniform random variable. We
first rewrite the identity (\ref{eqreccn}) in terms of the rescaled
quantities $(X_n)_{n\ge0}$: with $X_0 \equiv0$, almost surely,
%
\begin{eqnarray}\label{eqxns}
\qquad X_n(s) 
&= & \frac{\mu(I^{(0)}_n)}{\mu(n)}
\biggl[\mathbf{1}_{ \{ s \leq U
\} } X^{(0)}_{I_n^{(0)}} \biggl(
\frac{s}{1-(V-U)} \biggr) \nonumber
\\
&& \hspace*{38pt}{} + \mathbf{1}_{
\{ s > V \} } X^{(0)}_{I^{(0)}_n}
\biggl( \frac{s-(V-U)}{1-(V-U)} \biggr) \biggr]
\\
&&{}+\mathbf{1}_{ \{ U < s \leq V \} } \biggl[\frac{1} {\mu(n)}+ \frac
{\mu (I^{(0)}_n)}{\mu(n)}
X^{(0)}_{I^{(0)}_n} (\Psi ) + \frac
{\mu(I^{(1)}_n)}{\mu(n)}
X^{(1)}_{I^{(1)}_n} \biggl( \frac{s-U}{V-U} \biggr) \biggr].\nonumber
\end{eqnarray}
%
Here, $(X^{(0)}_n)_{n \geq0}$ and $(X^{(1)}_n)_{n \geq0}$ are defined
analogously to $(X_n)_{n \geq0}$ based on $(C_n^{(0)})_{n \geq0}$ and
$(C_n^{(1)})_{n \geq0}$, respectively.
The convergence of $X_n$ to $Z$ is naturally decomposed into two steps:
first the convergence of the coefficients of the recurrence relation in
(\ref{eqxns}), and second the contractive property of the limit recurrence.
In order to reflect this decomposition, we
define the accompanying sequence $(Q_n)_{n \geq0}$. Let $Q_0 \equiv0$
and for $n \geq1$,
%
\begin{eqnarray}
\label{eqqns}
\qquad Q_n(s)&:= & \frac{\mu(I^{(0)}_n)}{\mu(n)} \biggl[
\mathbf{1}_{ \{ s \leq U
\} } Z^{(0)} \biggl( \frac{s}{1-(V-U)} \biggr) \nonumber
\\
&&\hspace*{38pt}
{} + \mathbf{1}_{ \{ s > V \} } Z^{(0)} \biggl( \frac
{s-(V-U)}{1-(V-U)} \biggr)
\biggr]
\\
&&{}+ \mathbf{1}_{ \{ U < s \leq V \} } \biggl[\frac{1} {\mu(n)} + \frac
{\mu
(I^{(0)}_n)}{\mu(n)}
Z^{(0)} ( \Psi ) + \frac{\mu
(I^{(1)}_n)}{\mu(n)} Z^{(1)} \biggl(
\frac{s-U}{V-U} \biggr) \biggr].\nonumber
\end{eqnarray}
We first show that $\mathbf{E} [\|Q_n - Z\|^2] \rightarrow0$. A direct
application of the definition of~$Q_n$, its coupling with the process
$Z$,  and the characterization of $Z$ in Theorem~\ref{thmfixas} implies the
following bound for the supremum of $Q_n-Z$:
%
\begin{eqnarray}\label{eqqnzsup}
\|Q_n - Z\|  &\le&  \biggl( 3 \biggl\llvert
\frac{\mu(I^{(0)}_n)}{\mu(n)} - \bigl(1-(V-U)\bigr)^\beta\biggr\rrvert
\nonumber\\[-8pt]\\[-8pt]
&&\hspace*{26pt}{} + \biggl\llvert \frac{\mu(I^{(1)}_n)}{\mu(n)} - (V-U)^\beta\biggr\rrvert \biggr) \|Z\|
+\frac{1}{\mu(n)}.\nonumber
\end{eqnarray}
Here, the triangle inequality in $L^2$ is sufficient for our needs and
we obtain:
%
\begin{eqnarray}\label{eqqnz}
\bigl\| \|Q_n - Z\| \bigr\|_2 &\leq& \biggl( 3 \biggl\llVert \frac{\mu(I^{(0)}_n)}{\mu
(n)} - \bigl(1-(V-U)\bigr)^\beta\biggr\rrVert
_2
\nonumber\\[-8pt]\\[-8pt]
&&\hspace*{26pt}{}  + \biggl\llVert \frac{\mu
(I^{(1)}_n)}{\mu(n)} - (V-U)^\beta\biggr
\rrVert _2 \biggr) \bigl\|\| Z\| \bigr\|_2 + \frac{1}{\mu(n)}.\nonumber
\end{eqnarray}
By the asymptotic expansion of $\mu(n)$ in Theorem~\ref{thmunif} (actually,
$\mu(n) \sim c n^{\beta/2}$ as $n\to\infty$ is sufficient), it is easy
to see that
the term inside the bracket vanishes as $n \to\infty$. Since $\|Z\|$
is bounded
in $L^2$, this implies $\mathbf{E} [\|Q_n - Z\|^2] \rightarrow0$ as desired.

We now move on to showing that $\mathbf{E} [\|X_n-Z\|^2]\to0$ as
$n\to\infty$.

We will use the following properties that either hold true by
construction or
are easily checked by direct computations:
\begin{enumerate}[(O2)]
\item[(O2)] For any $n \in\mathbb{N}_0$, we have $\mathcal
{L}(X_n^{(0)},
Z^{(0)}) = \mathcal{L}(X_n^{(1)}, Z^{(1)}) = \mathcal{L}(X_n, Z)$
and $((X_n^{(0)})_{n\geq1}, Z^{(0)})$ and $((X_n^{(1)})_{n\geq1},
Z^{(1)})$ are independent.
\item[(O3)] For any $n \in\mathbb{N}$, the random variables $I^{(0)}_n,
\Psi, ((X_m^{(0)})_{m \geq1}, Z^{(0)})$ are independent. The same
holds for
$I^{(1)}_n, ((X_m^{(1)})_{m \geq1}, Z^{(1)})$.
\end{enumerate}
The Minkowski inequality in $L^2$ is not good enough anymore, and one
needs to develop the square and handle the terms separately. We have
%
\begin{eqnarray}\label{ungl10}
\quad && \mathbf{E} \bigl[\|X_n - Q_n\|^2\bigr]\nonumber
\\
&&\qquad \le
\mathbf{E} \biggl[ \biggl( \frac{\mu(I^{(0)}_n)}{\mu(n)} \biggr)^2 \cdot \bigl
\llVert X^{(0)}_{I_{n}^{(0)}} - Z^{(0)} \bigr\rrVert
^2 \biggr]
\nonumber\\[-8pt]\\[-8pt]
&&\quad\qquad{} + \mathbf{E} \biggl[ \biggl( \frac{\mu(I^{(1)}_n)}{\mu(n)}
\biggr)^2 \cdot\bigl\llVert X^{(1)}_{I_{n}^{(1)}} -
Z^{(1)} \bigr\rrVert ^2 \biggr]\nonumber
\\
&&\quad\qquad{} + 2 \mathbf{E} \biggl[\frac{\mu(I^{(0)}_n)}{\mu(n)} \frac{\mu
(I^{(1)}_n)}{\mu(n)} \bigl\llvert
X^{(0)}_{I^{(0)}_n} (\Psi )- Z^{(0)} ( \Psi ) \bigr\rrvert
\cdot\bigl\llVert X^{(1)}_{I_{n}^{(1)}} - Z^{(1)} \bigr\rrVert
\biggr].\nonumber
\end{eqnarray}
Using (O3), and the equalities in distribution for
$(X^{(i)}_m)_{m\ge1}$, $i=0,1$ with $(X_m)_{m\ge1}$, and
$(Z^{(i)}_m)_{m\ge1}$, $i=0,1$ with $(Z_m)_{m\ge1}$, this yields
%
\begin{eqnarray} \label{eqxnqn}
&& \mathbf{E} \bigl[\|X_n - Q_n\|^2\bigr]\nonumber
\\
&&\qquad \le
\biggl\{\mathbf{E} \biggl[ \biggl( \frac{\mu (I^{(0)}_n)}{\mu(n)} \biggr)^2
\biggr] + \mathbf{E} \biggl[ \biggl( \frac{\mu (I^{(1)}_n)}{\mu(n)} \biggr)^2
\biggr] \biggr\} \sup_{i<
n} \mathbf{E} \bigl[
\|X_{i} - Z\|^2\bigr]
\\
&&\quad\qquad{} + 2 \bigl\llVert X_{I^{(0)}_n} ( \Psi )- Z ( \Psi ) \bigr\rrVert
_2 \cdot\sup_{i < n} \bigl\llVert \llVert
X_i - Z\rrVert \bigr\rrVert _2.\nonumber
\end{eqnarray}
%
Let $\Delta(n):= \mathbf{E} [\|X_n - Z\|^2]$ and define
%
\begin{eqnarray}
\label{eqdef-Ln} L_n &=& \mathbf{E} \biggl[ \biggl( \frac{\mu(I^{(0)}_n)}{\mu(n)}
\biggr)^2 \biggr] + \mathbf{E} \biggl[ \biggl( \frac{\mu(I^{(1)}_n)}{\mu
(n)}
\biggr)^2 \biggr].
\end{eqnarray}

Since $I^{(0)}_n \uparrow\infty$ almost surely, and $I_n^{(0)}$ and
$\Psi$ are independent by (O3), Lemma~\ref{lemconvunif} implies
that as $n\to\infty$,
\[
\bigl\llVert X_{I_n^{(0)}} ( \Psi )- Z ( \Psi ) \bigr\rrVert _2
\to0.
\]
Let $\varepsilon_n$ be a sequence tending to zero as $n\to\infty$ such
that, for all $n\ge1$,
\[
\bigl\llVert X_{I_n^{(0)}} ( \Psi )- Z ( \Psi ) \bigr\rrVert _2
\le\varepsilon_n \quad\mbox{and}\quad\mathbf{E} \bigl[\|Q_n-Z
\| ^2\bigr]\le \varepsilon_n.
\]
By (O2), (O3) and the fact that $\Psi$ is uniformly distributed
on the unit interval,
Lemma~\ref{lemconvunif} together with (O2) and (\ref{eqxnqn}) yields
\[
\mathbf{E} \bigl[\|X_n - Q_n\|^2\bigr] \leq
L_n \cdot\sup_{i< n} \Delta(i) + 2 \varepsilon
_n \cdot\sup_{i< n} \Delta(i)^{1/2}.
\]
Altogether, we have for every $n\ge1$
%
\begin{eqnarray}
\Delta(n) &\leq& \mathbf{E} \bigl[\|X_n - Q_n
\|^2\bigr] + \mathbf{E} \bigl[\|Q_n - Z \|^2
\bigr]
\nonumber\\[-8pt]\label{intriangle0} \\[-8pt]
&&{} + 2 \sqrt{\mathbf{E} \bigl[\| X_n - Q_n
\|^2\bigr] \cdot\mathbf{E} \bigl[\| Q_n - Z \|
^2\bigr]}\nonumber
\\
&\leq& L_n \cdot\sup_{i < n} \Delta(i) + 2
\varepsilon_n \cdot\sup_{i
<n} \Delta(i)^{1/2}
+ \varepsilon_n
\nonumber\\[-8pt]\label{intriangle} \\[-8pt]
&&{} + 2 \sqrt{\varepsilon_n} \Bigl( L_n
\cdot \sup_{i < n} \Delta(i) + \varepsilon_n \cdot\sup
_{i < n} \Delta(i)^{1/2} \Bigr)^{1/2}.\nonumber
\end{eqnarray}
Now, by the bounded convergence theorem, $L_n\to q:=\mathbf{E}
[(1-(V-U))^{2\beta}] + \mathbf{E} [(V-U)^{2\beta}]$ as $n\to\infty
$. Thus,
since $2\beta>1$, $L_n$ eventually drops below one for $n$ sufficiently
large, and it easily follows that $\Delta(n)$ is bounded.

To prove that $\Delta(n)\to0$, let $K:= \sup_n \Delta(n)$ and $a:=
\limsup_n \Delta(n)$.
Then let $\delta> 0$ be arbitrary and choose $\ell$ large enough such that
$\Delta(n) \leq a + \delta$ for $n \geq\ell$. This $\ell$ being fixed,
let $n_0\ge\ell$ be large enough such that for $n\ge n_0$ one has
${\mathbf P}(I_n^{(i)} \le\ell) < \delta/ K$ for $i = 0, 1$. Then,
combining (\ref{intriangle0}) and the bound (\ref{ungl10}),
conditioning on the value
of $I_n^{(0)}$ and $I_n^{(1)}$, respectively, and splitting the
integrand into the cases $\{I_n^{(0)} \leq\ell\}$ and $\{I_n^{(0)} >
\ell\}$ and similarly for $I_n^{(1)}$, we obtain for all $n\ge n_0$
%
\[
\Delta(n) \leq2 \delta+ L_n (a+ \delta) + 2 \varepsilon_n
K^{1/2} + \varepsilon_n + 2 \sqrt{\varepsilon_n}
\bigl( L_n \cdot K + \varepsilon_n \cdot K^{1/2}
\bigr)^{1/2}.
\]
%
First letting $n \to\infty$ and then $\delta\downarrow0$, we obtain
$a \leq q a$. The fact that $q <1$ implies that $a = 0$, so that
$\Delta
(n)=\mathbf{E} [\|X_n-Z\|^2]\to0$ as $n\to\infty$, which completes
the proof.
\end{pf*}

Finally, it remains to prove the convergence at a uniform point stated
in Lemma~\ref{lemconvunif}, which is the true cornerstone of our
argument. 
%
\begin{pf*}{Proof of Lemma~\ref{lemconvunif}}
We proceed along the same lines as in the process case relying on
arguments that have already been used in the construction of the limit
in Section~\ref{seclimit}. First, we clearly have $\|Q_n(\xi) -
Z(\xi)
\|_2 \rightarrow0$ as the term is bounded by $\| \|Q_n - Z\| \|_2$
which was shown to vanish asymptotically in (\ref{eqqnz}).
Let $W_0 = K_0(\xi, U,V)$ and $W_1 = K_1(\xi, U,V)$. Then, by the
recursions (\ref{eqxns}) and (\ref{eqqns}) for $X_n(s)$ and $Q_ n(s)$
taken at $s=\xi$, we have
%
\begin{eqnarray}\label{xnqn3}
&&\mathbf{E} \bigl[\bigl|X_n(\xi) - Q_n(\xi)\bigr|^2\bigr]\nonumber
\\
&&\qquad \le\mathbf{E} \biggl[ \biggl( \frac{\mu(I^{(0)}_n)}{\mu(n)} \biggr)^2
\bigl(X^{(0)}_{I_{n}^{(0)}}(W_0) - Z^{(0)}
(W_0) \bigr) ^2 \biggr]\nonumber
\\
&&\quad\qquad{}  + \mathbf{E} \biggl[ \biggl(
\frac{\mu(I^{(1)}_n)}{\mu(n)} \biggr)^2 \bigl(X^{(1)}_{I_{n}^{(1)}}(W_1)
- Z^{(1)}(W_1) \bigr)^2 \biggr]
\\
&&\quad\qquad {}+ 2\mathbf{E} \biggl[\mathbf{1}_{ \{ U \leq\xi< V \} } \frac{\mu(I^{(0)}_n)}{\mu(n)}
\frac {\mu(I^{(1)}_n)}{\mu(n)}\nonumber
\\
&&\hspace*{61pt}{}\times  \bigl(X^{(0)}_{I^{(0)}_n} ( \Psi )-
Z^{(0)} ( \Psi ) \bigr) \bigl(X^{(1)}_{I_{n}^{(1)}}(W_1)
- Z^{(1)}(W_1) \bigr) \biggr].\nonumber 
\end{eqnarray}
To handle these terms, we use another property which can be seen as an
extension of (O3).
\begin{enumerate}[(O4)]
\item[(O4)] For any $n \in\mathbb{N}$, we have independence of
$I^{(0)}_n, W_0, ((X^{(0)}_m)_{m\geq1}, Z^{(0)})$. 
Moreover, on $\{U < \xi< V\}$, the quantities $(I^{(0)}_n, I^{(1)}_n)$
and $((X_m^{(0)})_{m \geq1},\break  Z^{(0)})$, $((X_m^{(1)})_{ m\geq1},
Z^{(1)}$), $\Psi, W_1$ are independent.
\end{enumerate}
Using (O2) and (O4), conditioned on the values of $I^{(0)}_n$
and $I^{(1)}_n$,
one sees that the mixed term in (\ref{xnqn3}) vanishes since
$\mathbf{E} [X_n(\xi)] = \mathbf{E} [Z(\xi)] = \kappa\mathrm
{B}(\beta+ 1, \beta
+1)$ for all
$n \geq1$ and $\mu(0) = 0$.

From there, again using (O2), (O4) and the fact that $W_0, W_1$
are uniformly distributed on the unit interval for the first two terms
in (\ref{xnqn3}), we see that
%
\begin{equation}
\label{eqrecuniformLn} \mathbf{E} \bigl[\bigl|X_n(\xi) - Q_n(
\xi)\bigr|^2\bigr] \leq L_n \cdot\sup_{i <n}
\mathbf{E} \bigl[\bigl|X_{i}(\xi ) - Z (\xi) \bigr|^2\bigr],
\end{equation}
where $L_n$ is the quantity defined in (\ref{eqdef-Ln}).
As before, (\ref{eqrecuniformLn}) above implies that the sequence
$\mathbf{E} [|X_n(\xi)-Z(\xi)|^2]$ is bounded. The claim then
follows from the
same arguments (even simpler) as at the end of the proof of
Theorem~\ref{thmconvunif}, starting from~(\ref{xnqn3}) and using again the fact
that the mixed term there equals zero. We omit the details.
\end{pf*}

%
\begin{rem*}
The fact that the mixed terms in (\ref{xnqn3}) above vanish
is crucial since at this point we have otherwise no control on the
first moment. Moreover, the mixed terms vanish only when looking at the
uniform location $\xi$: more precisely, one could not use this argument
directly at a fixed location $s$ because for fixed $s$, $\Psi$ and
$(s-U)/(V-U)$ are not independent. In other words, there is no obvious
shortcut in our argument and it seems that there is no way around
showing convergence at a uniform point first.
\end{rem*}

\subsection{Uniform convergence in $L^m$, \texorpdfstring{$m\ge2$}{m>=2} and almost sure convergence}\label{secLm}

%
\begin{cor} \label{corhighermom}
For any $m \in\mathbb{N}$, we have $\mathbf{E}  [\|X_n - Z\|^m
] \rightarrow0$.
\end{cor}

\begin{pf} 
First note that Theorem~\ref{thmconvunif} implies $\|X_n - Z \|
\rightarrow0$ in probability.
Moreover, by Theorem~\ref{thmfixas}, $\mathbf{E}  [\|Z\|^m
] < \infty$ for all $m
\geq1$. The inductive argument used to prove $\sup_{n \geq1} \mathbf
{E}  [\| Z_n\|^m  ] < \infty$ for all $m$ based on
inequality (\ref{ineqhighermoments}) can be worked out similarly to
show $\sup_{n \geq1} \mathbf{E}  [\|X_n\|^m  ] < \infty$
for all $m$, and we omit
the details.
\end{pf}
%

Taking
more care on the error terms in the proof of Theorem~\ref{thmconvunif}, we can prove the following rate of convergence in
$L^2$, which is the
key to the proof of the almost sure uniform convergence of $X_n$ to
$Z$. Here and subsequently, we use the big-O Landau symbols for
sequences of time parameter $n$ as $n \to\infty$.

\begin{lem}\label{lemconvunifrate}
Let $\xi$ be uniform in $[0,1]$, and independent of $(X_n)_{n\ge1}$
and $Z$. Then, for any $\kappa< 2 \beta-1$, we have $\mathbf{E}
[|X_n(\xi) - Z(\xi)|^2] = O(n^{-\kappa})$ as $n\to\infty$.
\end{lem}

\begin{pf}Let $\operatorname{Bin}(n,p)$ have binomial distribution
with parameters
$n,p$. Using standard concentration results for the binomial
distribution, it is easy to see that for any $\varrho> 0$,
%
\begin{equation}
\label{boundcher} \mathbf{E} \biggl[\biggl\llvert \frac{\operatorname{Bin}(n,p)}{n} - p \biggr
\rrvert ^{\varrho } \biggr] = O\bigl(n^{-\varrho/2}\bigr),
\end{equation}
uniformly in $p \in[0,1]$. The difference between the limit $Z$ and
the accompanying sequence $Q_n$ is easily bounded: first using the fact
that $|x^\beta- y^\beta| \leq|x-y|^\beta$ in the right-hand side of
(\ref{eqqnz}) and then using the bound (\ref{boundcher}), we obtain
%
\begin{equation}
\label{eqboundqnzrate} \mathbf{E} \bigl[\|Q_n - Z \|^2\bigr] = O
\bigl(n^{-\beta}\bigr).
\end{equation}

Let $d(n) = \mathbf{E} [|X_n(\xi) - Z(\xi)|^2]$. Then, using (\ref{xnqn3}), we
obtain (recall that the mixed term equals zero)
\begin{eqnarray*}
d(n) &\leq& \mathbf{E} \bigl[\bigl|X_n(\xi) - Q_n(
\xi)\bigr|^2\bigr] + \mathbf{E} \bigl[\bigl|Q_n(\xi ) -Z(
\xi)\bigr|^2\bigr]
\\
&&{} + 2 \bigl\| X_n(\xi) - Q_n(\xi)
\bigr\|_2 \cdot\bigl\|Q_n(\xi) - Z(\xi)\bigr\|_2
\\
&\leq& \mathbf{E} \biggl[ \biggl( \frac{\mu(I^{(0)}_n)}{\mu(n)} \biggr)^2 d
\bigl(I_n^{(0)}\bigr) \biggr] + \mathbf{E} \biggl[ \biggl(
\frac{\mu(I^{(1)}_n)}{\mu(n)} \biggr)^2 d\bigl(I_n^{(1)}
\bigr) \biggr]
\\
&&{} + O\bigl(\bigl\| X_n(\xi) - Q_n(\xi)
\bigr\|_2 \cdot n^{-\beta/2}\bigr) + O\bigl( n^{-\beta}\bigr).
\end{eqnarray*}
Fix $\kappa< 2\beta- 1$ and let $r_n:=\sup\{d(i) \cdot(i\vee
1)^\kappa\dvtx  0\le i\le n-1\}$.
Then the inequality above implies that for all $n$ we have
\[
d(n) \leq r_n n^{-\kappa} \ell_n +
C_3 n^{-\beta/2}
\]
for some constant $C_3$ and
%
\begin{equation}
\label{eqdefln} \qquad \ell_n:=\mathbf{E} \biggl[ \biggl( \frac{\mu(I^{(0)}_n)}{\mu
(n)}
\biggr)^2 \biggl( \frac{I_n^{(0)}\vee1}{n} \biggr)^{-\kappa} \biggr] +
\mathbf{E} \biggl[ \biggl( \frac {\mu(I^{(1)}_n)}{\mu
(n)} \biggr)^2 \biggl(
\frac{ I_n^{(1)}\vee 1}{n} \biggr)^{-\kappa} \biggr].
\end{equation}
By the bounded convergence theorem, one has as $n\to\infty$
\[
\ell_n \to\mathbf{E} \bigl[\bigl(1-(V-U)\bigr)^{2\beta-\kappa} \bigr]
+ \mathbf{E} \bigl[(V-U)^{2\beta -\kappa} \bigr]<1,
\]
since our choice for $\kappa$ ensures that $2\beta-\kappa>1$.
Thus, there exists $\gamma>0$ and $n_0$ large enough such that for all
$n\ge n_0$ one has $\ell_n\le1-\gamma$. Now, let $n_1\ge n_0$ be large
enough such that $r_{n_0} \gamma n^{-\kappa}> C_3 n^{-\beta/2}$ for all
$n\ge n_1$, which is possible since $2\beta-1<\beta/2$. An easy
induction on $n$ then shows that for all $n\ge n_1$ we have
\[
d(n)=\mathbf{E} \bigl[\bigl|X_n(\xi)-Z(\xi)\bigr|^2\bigr] \le
r_{n_0} n^{-\kappa}
\]
as desired.
\end{pf}

%
\begin{lem} \label{lemxnzL2}
For any $\kappa< 2 \beta-1$, we have $\mathbf{E} [\|X_n - Z\|^2] =
O(n^{-\kappa})$.
\end{lem}

\begin{pf}
As in the proof of Theorem~\ref{thmconvunif}, we abbreviate $\Delta(n)
= \mathbf{E} [\|X_n - Z\|^2]$ and recall inequality (\ref{intriangle0})
%
\begin{eqnarray} \label{intriangle01}
\mathbf{E} \bigl[\|X_n-Z\|^2\bigr] 
&\le&\mathbf{E} \bigl[\|X_n-Q_n
\|^2\bigr]+ \mathbf{E} \bigl[\|Q_n-Z\|^2
\bigr]
\nonumber\\[-8pt]\\[-8pt]
&&{}   + 2 \sqrt {\mathbf{E} \bigl[\|X_n-Q_n
\|^2\bigr] \cdot\mathbf{E} \bigl[\|Q_n-Z\|^2
\bigr]}.\nonumber
\end{eqnarray}
%
We have already seen in the course of the proof of Lemma~\ref{lemconvunifrate}, equation (\ref{eqboundqnzrate}) that $\mathbf
{E} [\| Q_n - Z \|^2] = O(n^{-\beta})$. To bound the terms involving
$\mathbf{E} [\| X_n-Q_n\|^2]$, we use equation (\ref{ungl10}) from
the proof of
Theorem~\ref{thmconvunif}.

Fix $\kappa< 2\beta- 1$ and $\kappa'$ such that $\kappa< \kappa' <
2\beta-1$. Combining (\ref{intriangle01}), (\ref{ungl10}), using
Cauchy--Schwarz inequality to decouple the mixed term and applying
Lemma~\ref{lemconvunifrate}, we obtain
\begin{eqnarray*}
\Delta(n) &\leq& \mathbf{E} \biggl[ \biggl( \frac{\mu(I^{(0)}_n)}{\mu(n)}
\biggr)^2 \Delta \bigl(I_n^{(0)}\bigr) \biggr] +
\mathbf{E} \biggl[ \biggl( \frac{\mu(I^{(1)}_n)}{\mu(n)} \biggr)^2 \Delta
\bigl(I_n^{(1)}\bigr) \biggr]
\\
&&{} + O \Bigl( \sqrt{\mathbf{E} \bigl[\Delta\bigl(I_n^{(1)}
\bigr)\bigr] \cdot\mathbf{E} \bigl[\bigl(I_n^{(0)}\vee 1
\bigr)^{-\kappa'}\bigr]} \Bigr)
\\
&&{}  + O\bigl( \sqrt{\mathbf{E} \bigl[\|
X_n - Q_n \|^2\bigr]} \cdot n^{-\beta/2}
\bigr)+ O\bigl(n^{-\beta}\bigr).
\end{eqnarray*}
This recurrence relation is almost identical to that in the proof of
Lemma~\ref{lemconvunifrate}, and we only indicate how to deal with
the extra term coming from the mixed term of~(\ref{ungl10}). Write
$R_n:=\sup\{\Delta(i)\cdot(i\vee1)^\kappa\dvtx  i<n\}$. Then
\[
\mathbf{E} \bigl[\Delta\bigl(I_n^{(1)}\bigr)\bigr] \le
R_n n^{-\kappa} \mathbf{E} \biggl[ \biggl(\frac {I_n^{(1)}\vee 1}n
\biggr)^{-\kappa} \biggr].
\]
Since $\kappa,\kappa'<1$, a standard application of a truncation
argument, bounded and monotone convergence theorems imply that there
exists a constant $C_4$ such that 
%
\[
\mathbf{E} \biggl[ \biggl(\frac{I_n^{(1)}\vee1}n \biggr)^{-\kappa} \biggr]\le
C_4 \quad \mbox{and}\quad\mathbf{E} \biggl[ \biggl(\frac{I_n^{(0)}\vee
1}n
\biggr)^{-\kappa'} \biggr]\le C_4.
\]
Thus, there exists a constant $C_5$ such that for all $n$ large enough,
\[
\Delta(n) \le R_n \ell_n n^{-\kappa} +
C_5 n^{-\beta/2} + C_5\sqrt{R_n}
n^{-(\kappa+ \kappa')/2},
\]
where $\ell_n$ is the quantity defined in (\ref{eqdefln}).
From here, the claim that $\Delta(n) = O (n^{-\kappa})$ follows by yet
another induction on $n$ using the same arguments as above, and we omit
the details.
\end{pf}

%
\begin{prop} \label{propratetransprocess} For any $m\ge1$, and for
any $\delta<\beta-1/2$, we have $\mathbf{E} [\|X_n-Z\|^m]=O(n^{-m
\delta})$.
\end{prop}

\begin{pf}Again, we introduce the intermediate sequence $(Q_n)_{n\ge
1}$. We proceed by induction on $m\ge1$. Lemma~\ref{lemxnzL2}
implies that $\|\|X_n-Z\|\|_1\le\|\|X_n-Z\|\|_2 =O(n^{-\kappa/2})$,
for any $\kappa<2\beta-1$ so that the claim holds for the first two
moments. We suppose now that for every $k<m$, and every $\delta\in
(0,\beta-1/2)$, $\mathbf{E} [\|X_n-Z\|^k]=O(n^{-k \delta})$; so we
will prove
the claim for all $\delta$ at once.

Now fix $\delta\in(0,\beta-1/2)$, and pick $\eta\in(\delta, \beta
-1/2)$. Note that the arguments used to prove that $\mathbf{E} [\|
Q_n-Z\| ^2]=O(n^{-\beta})$ also yield that for any $k\ge1$,
$\mathbf{E} [\|Q_n-Z\|^k]=O(n^{-k\beta/2})$. 
Write $\Delta_k(n):=\mathbf{E} [\|X_n-Z\|^k]$. By the induction hypothesis,
there exists constants $K_k$, $k=1,\ldots,m-1$, such that $\Delta
_k(n)\le
K_k (n\vee1)^{-k\eta}$ for every $n\ge0$.

Then, expanding the moments using the bound $\|X_n-Z\|\le\|X_n-Q_n\|
+\|Q_n-Z\|$, we obtain
\begin{eqnarray*}
\Delta_m(n) &\le&\sum
_{k=0}^m \binom m k \cdot\mathbf{E} \bigl[
\|X_n-Q_n\| ^{m-k} \cdot\| Q_n-Z\|
^{k} \bigr]
\\
&\le&\sum_{k=0}^m \binom m k \cdot
\mathbf{E} \bigl[\|X_n-Q_n\|^m
\bigr]^{1-k/m} \cdot\mathbf{E} \bigl[\| Q_n-Z\|^m
\bigr]^{k/m},
\end{eqnarray*}
where the second inequality follows from H\"older's inequality (for
$1\le k<m$, but one sees that the inequality is also valid when $k=0$
or $k=m$). Therefore, we have, for some constant $C_6$,
%
\begin{equation}
\label{eqrecDeltamn} \Delta_m(n) \le C_6 \cdot\sum
_{k=0}^m \binom m k n^{-k\beta/2} \cdot
\mathbf{E} \bigl[\|X_n-Q_n\|^m
\bigr]^{1-k/m}.
\end{equation}

We now reexpress the term $\mathbf{E} [\|X_n-Q_n\|^m]$ in terms of
$\Delta
_{k}$, $k\in\{1,\ldots, m\}$ as follows:
\begin{eqnarray*}
&& \mathbf{E} \bigl[\|X_n - Q_n\|^m\bigr]
\\
&&\qquad \le \mathbf{E} \biggl[ \biggl( \frac{\mu(I^{(0)}_n)}{\mu(n)} \biggr)^m \bigl\llVert
X^{(0)}_{I_{n}^{(0)}} - Z^{(0)} \bigr\rrVert ^m
\biggr] + \mathbf{E} \biggl[ \biggl( \frac{\mu (I^{(1)}_n)}{\mu(n)} \biggr)^m
\bigl\llVert X^{(1)}_{I_{n}^{(1)}} - Z^{(1)}\bigr\rrVert
^m \biggr]
\\
&&\quad\qquad{} + \sum_{k=1}^{m-1} \binom m k \mathbf{E}
\biggl[ \biggl( \frac{\mu
(I_n^{(0)})}{\mu(n)} \bigl\llVert X^{(0)}_{I^{(0)}_n}-
Z^{(0)}\bigr\rrVert \biggr)^k
\\
&&\hspace*{102pt}{}\times  \biggl(
\frac{\mu(I_n^{(1)})}{\mu(n)} \bigl\llVert X^{(1)}_{I_{n}^{(1)}} -
Z^{(1)} \bigr\rrVert \biggr)^{m-k} \biggr]
\\
&&\qquad \le \mathbf{E} \biggl[ \biggl( \frac{\mu(I_n^{(0)})}{\mu(n)} \biggr)^m \Delta
_{m}\bigl(I_n^{(0)}\bigr) \biggr] + \mathbf{E}
\biggl[ \biggl( \frac{\mu(I_n^{(1)})}{\mu(n)} \biggr)^m \Delta _{m}
\bigl(I_n^{(1)}\bigr) \biggr]
\\
&&\quad\qquad{}+ \sum_{k=1}^{m-1} \binom m k \mathbf{E}
\biggl[ \biggl( \frac{\mu
(I_n^{(0)})}{\mu(n)} \biggr)^k \Delta_k
\bigl(I_n^{(0)}\bigr) \cdot \biggl( \frac{\mu(I_n^{(1)})}{\mu(n)}
\biggr)^{m-k}\Delta_{m-k}\bigl(I_n^{(1)}
\bigr) \biggr].
\end{eqnarray*}
The induction hypothesis then implies that, for every $n\ge0$, the
last sum above is at most
\begin{eqnarray*}
&& n^{-m \eta} \sum_{k=1}^{m-1} \binom m
k K_k K_{m-k} \mathbf{E} \biggl[ \biggl( \frac {\mu(I_n^{(0)})}{\mu(n)}
\biggl(\frac{I_n^{(0)}\vee
1}n \biggr)^{-\eta} \biggr)^k
\\
&&\hspace*{121pt}{}\times
\biggl( \frac{\mu
(I_n^{(1)})}{\mu(n)} \biggl(\frac {I_n^{(1)}\vee1}{n} \biggr)^{-\eta
}
\biggr)^{m-k} \biggr].
\end{eqnarray*}
But since $\mu(n)\sim c n^{-\beta/2}$ and $\eta<\beta-1/2<\beta
/2$, the
almost sure convergence of $(I_n^{(0)}/n, I_n^{(1)}/n)$ implies that
the expected values above are all bounded, which implies that one
actually has
\begin{eqnarray*}
&& \mathbf{E} \bigl[\|X_n-Q_n\|^m\bigr]
\\
&&\qquad \le
\mathbf{E} \biggl[ \biggl( \frac{\mu
(I_n^{(0)})}{\mu (n)} \biggr)^m
\Delta_{m}\bigl(I_n^{(0)}\bigr) \biggr] +
\mathbf{E} \biggl[ \biggl( \frac{\mu(I_n^{(1)})}{\mu (n)} \biggr)^m
\Delta_{m}\bigl(I_n^{(1)}\bigr) \biggr] +
C_7 n^{-m\eta}
\end{eqnarray*}
for some constant $C_7$. Let $R_n^{(m)}:=\sup\{ i^{m\delta} \Delta_m(i)
\dvtx  i< n\}$ so that we have
\[
\mathbf{E} \bigl[\|X_n-Q_n\|^m\bigr] \le
R^{(m)}_n n^{-m \delta} \ell_n^{(m)}
+ C_7 n^{-m
\eta},
\]
where, as before,
\[
\ell_n^{(m)}:=\mathbf{E} \biggl[ \biggl(
\frac{\mu(I^{(0)}_n)}{\mu
(n)} \biggl( \frac{I_n^{(0)}\vee1}{n} \biggr)^{-\delta}
\biggr)^m \biggr] + \mathbf{E} \biggl[ \biggl( \frac{\mu(I^{(1)}_n)}{\mu(n)}
\biggl( \frac{ I_n^{(1)}\vee 1}{n} \biggr)^{-\delta} \biggr)^m \biggr]<1-
\gamma
\]
for some $\gamma>0$ and all $n$ large enough.

From there, the same arguments we used before allow us to treat the
recurrence relation in (\ref{eqrecDeltamn}), and to conclude that
$C_n$ is actually bounded. We omit the details, but just note that
although the main term in the right-hand side of (\ref{eqrecDeltamn}) is the one for $k=0$, the others cannot be dropped
earlier or one would not be able to prove a rate better than $n^{-\beta
/2}$, regardless of $m$.
\end{pf}

Let $m$ be large enough such that $m (\beta-1/2)>2$. Then by Markov's
inequality, for any $\varepsilon>0$ and all $n$ large enough, we have
\[
{\mathbf P} \bigl(\|X_n-Z\|>\varepsilon \bigr)\le\varepsilon^{-m} \cdot
\mathbf{E} \bigl[\|X_n-Z\|^m\bigr] \le
\varepsilon^{-m} n^{-2}.
\]
It follows that, for any $\varepsilon>0$, we have
$\sum_{n\ge1} {\mathbf P} (\|X_n-Z\|>\varepsilon )<\infty$,
so that by the Borel--Cantelli lemma $\|X_n-Z\|\to0$ almost surely as
$n\to\infty$.
Together with Theorem~\ref{thmfixas}, Proposition~\ref{propunique}
and Corollary~\ref{corhighermom} this shows Theorem~\ref{thmmain} and
the first part of Theorem~\ref{thmdualtree}.

%
\begin{rem*}
In order to obtain almost sure convergence of $n^{-\beta/2} C_n$ rather
than convergence in probability we have transferred rates of
convergence for the coefficients in the recursive decomposition to the
convergence of the sequence of interest by induction. This is a
standard approach in the context of the contraction method particularly
in function spaces, where convergence rates (with respect to more
elaborate probability metrics) are necessary in order to deduce
functional limit theorems on a distributional level; see \cite{NeSu12}
for details.
\end{rem*}

\subsection{Convergence of the lamination} \label{secconvlam}
In this section, we complete the proof of Theorem~\ref{thmdualtree} by
showing that the process convergence in Theorem~\ref{thmmain} implies
convergence of the lamination $\mathfrak L_n$ to $\mathfrak L_Z$. Note
that, in
general, it is not sufficient that $f_n\to f$ uniformly on $[0,1]$ for
$\mathfrak L_{f_n}$ to converge to $\mathfrak L_f$, and we need
additional arguments.
We recall from \cite{legalcu}, Definition 2.1, that a lamination
$\mathfrak L
$ is called \emph{maximal} if for any $x,y \in\mathscr C$ with
$\llbracket x,y \rrbracket
\notin\mathfrak L$, the chord $\llbracket x,y \rrbracket$ intersects
at least one of the
chords in $\mathfrak L$. In other words, $\mathfrak L$ cannot be
enlarged by the
addition of other chords.
Le Gall and Paulin \cite{LePa2008a} have proved that for a geodesic
lamination of the hyperbolic disk encoded by a continuous function $g$
to be maximal it suffices that $g$ has distinct local minima on the
open interval $(0,1)$; the setting here is not exactly the same, but
the statement is easily adapted and we omit the details (see also
Proposition~2.5 in \cite{legalcu}). Maximal laminations $\mathfrak L$ coincide
with triangulations of the disk, that is laminations in which every
connected component of $\mathscr D\setminus\mathfrak L$ is an open
triangle whose
endpoints lie on the circle $\mathscr C$.

Let $\mathbb{L}$ denote the set of laminations of the disk which
contain only finitely many chords and satisfy the additional properties
that no chord has zero as an endpoint and that distinct chords do not
share a common endpoint.

%
\begin{lem}\label{lemconvlam}
Let $(\mathfrak L_n)_{n \geq0}$ be an increasing sequence in $\mathbb{L}$,
and let $f_n$ be a~function encoding $\mathfrak L_n$ in the sense that
$\mathfrak L
_n=\mathfrak L_{f_n}$. Suppose that $f_n\to g$ uniformly as $n \to
\infty$
where $g$ is continuous on $[0,1]$. Then
%
\begin{equation}
\label{converse} \mathfrak L_g \subseteq\overline {\bigcup
_{n
\geq1} \mathfrak L_n}.
\end{equation}
Moreover, if $\mathfrak L_g$ is maximal, then $\mathfrak L_n\to
\mathfrak L_g$ as $n\to\infty
$ for the Hausdorff metric.
\end{lem}

One can certainly not have $\mathfrak L_g=\lim_n \mathfrak L_n$
without any
additional assumption such as maximality. To see this, consider, for
instance, the case in which the scaling factors used to ensure
convergence of $f_n$ grow too fast so that $g(s)=0$ for all $s\in
[0,1]$; then the lamination $\mathfrak L_g$ is actually always empty.

We do not have a short argument why local minima of the limit function
$Z$ are almost surely distinct. Thus, we refer to the \cite{legalcu}, Corollary~5.3, for a direct proof of the inclusion
%
\[
\overline{\bigcup_{n \geq1} \mathfrak L_n}
\subseteq\mathfrak L_g.
\]
Together with Lemma~\ref{lemconvlam}, this finishes the proof of
Theorem~\ref{thmdualtree}.
\begin{pf*}{Proof of Lemma~\ref{lemconvlam}}
We first show that $\mathfrak L_g \subseteq\overline{\bigcup_{n \geq
1} \mathfrak L_n}$.
Consider a chord $\llbracket x,y \rrbracket\subset\mathfrak L_g$
with $x< y$ and assume that it is compatible with $g$. By definition
and continuity of $g$, one has $g(s) > g(x)=g(y)$ for all $s \in
(x,y)$. Let $0 < \varepsilon< (y-x)/3$ and $\delta= \inf_{s \in
[x+\varepsilon, y - \varepsilon]} g(s) - g(x)$. Choose $n$ large enough such
that $\|f_n - g \| < \delta/6$. Then, pick $\gamma\in(1/6,1/3)$ such
that at every point of discontinuity $s\in[x,y]$ of $f_n$, we have
$f_n(s)\neq g(x)+(1-\gamma)\delta$; this is possible since $f_n$ has at
most finitely many jumps.
For all $s\in[x+\varepsilon, y-\varepsilon]$, we have $f_n(s)>g(x)+\delta
(1-\gamma)$ and $f_n(x),f_n(y)<g(x)+\gamma\delta$. Let 
%
\begin{eqnarray*}
a_n&=:=&\inf\bigl\{a\le x+\varepsilon\dvtx  f_n(s)>g(x)+(1-
\gamma)\delta\ \forall s\in\bigl(a,x+\varepsilon]\bigr\},
\\
b_n&=:=&\sup\bigl\{b\ge y-\varepsilon\dvtx  f_n(s)>g(x)+(1-
\gamma)\delta\  \forall s\in[y-\varepsilon, b\bigr)\bigr\}.
\end{eqnarray*}
%
Then $a_n\in[x,x+\varepsilon]$, $b_n\in[y-\varepsilon, y]$ and for all
$s\in(a_n,b_n)$ we have $f_n(s)>g(x)+(1-\gamma)\delta$ and $\max\{
f_n(a_n),f_n(b_n)\}\le g(x)+(1-\gamma)\delta$. By the choice of~$\gamma
$, we have $f_n(a_n-)=f_n(a_n)$, and $\llbracket a_n,b_n \rrbracket$ is
$f_n$-compatible. It follows that $\llbracket a_n,b_n \rrbracket
\subset\mathfrak L_n$ and
that, by construction,
%
\[
\llbracket x,y \rrbracket\subset\bigcup_{n \geq1}
\mathfrak L_n^{2\varepsilon},
\]
where we recall that $A^\varepsilon$ denotes the $\varepsilon$-fattening of
$A$ in $\mathscr D$, $\{x\in\mathscr D\dvtx  \exists a\in A$ and
$|x-a|<\varepsilon\}$.
Letting $\varepsilon\downarrow0$ shows
$\llbracket x,y \rrbracket\subset\overline{\bigcup_{n \geq1}
\mathfrak L_n}$. In the case
that $\llbracket x,y \rrbracket$ is a limit of compatible chords one
applies a similar
argument to the sequence of \mbox{$g$-}compatible chords $(\llbracket x_k,y_k
\rrbracket, k\ge
1)$ such that $\llbracket x_k,y_k \rrbracket\to\llbracket x,y
\rrbracket$.
Together, this gives $\mathfrak L_g \subseteq\overline{\bigcup_{n
\geq1} \mathfrak L_n}$.

If $\mathfrak L_g$ is maximal, then we cannot have $\mathfrak
L_g\subsetneq\overline
{\bigcup_{n\ge1} \mathfrak L_n}$, since $\overline{\bigcup_{n\ge1}
\mathfrak L_n}$
is indeed a lamination. It follows that $\mathfrak L_g= \overline
{\bigcup_{n\ge1} \mathfrak L_n}$, which completes the proof.
\end{pf*}

\section{The dual of the homogeneous lamination} \label{seclimith}

In this section, we treat the case of the homogeneous lamination
process, in which the chords are added to a uniformly random fragment,
regardless of its mass. In this case, where the index of the
fragmentation is $\alpha=0$, the limit process is different from $Z$,
which is the common limit process to all fragmentations with a positive
index $\alpha$ \cite{legalcu}.

Let us first give a precise description of the process.
As before, $(U_v, V_v)_{v \in\mathcal T}$ denotes a collection of
independent random
variables with density\break $2\mathbf{1}_{ \{ 0< u <v <1 \} }$.
Independently of this set,
let $(J_n)_{n \geq1}$ be a sequence of independent random variables where,
for each $i \geq1$, $J_i$ has uniform distribution on $\{1, \ldots,
i\}$.
For $n =1$, insert $U = U_\varnothing$, $V = V_\varnothing$ and split the
disk into
fragments $S^{(0)}, S^{(1)}$ just as in the case $\alpha=2$. At time $n$,
we choose an arbitrary labeling of the $n$ different available
fragments $S^{(v_1)},\ldots, S^{(v_n)}$ and insert a chord in the fragment
$S^{(v_{J_n})}$. Here, writing $c$ for the mass of the fragment (the
one-dimensional Lebesgue measure of its intersection with the circle),
the insertion
is performed by choosing the endpoints to be given by the vector $(cU_{v_{J_n}},
cV_{v_{J_n}})$ with respect to the fragment (the origin
of the local coordinates is placed at the point corresponding to the
unique chord which separates $S^{(v_{J_n})}$ from its ancestor in the
dual tree).


The recursive decomposition for $C_n^h$ looks as in the case of $\alpha
= 2$, only the splitting random variable $(I_n^{(0)}, I_n^{(1)})$ has a
different distribution. [We use the same notation for the pair
$(I_n^{(0)}, I_n^{(1)})$ as in the\vspace*{1.5pt} case $\alpha=2$ for the sake of
readability.] With
$(C_n^{h,(0)})_{n \geq0}, (C_n^{h,(1)})_{n \geq0}$ defined
analogously to the case $\alpha= 2$ (remember the beginning of
Section~\ref{secconv} for details), we have for every $s\in[0,1]$
%
\begin{equation}\label{eqreccn2}
C_n^h(s)= \cases{
\displaystyle C^{h,(0)}_{I^{(0)}_n} \biggl( \frac{s}{1-(V-U)} \biggr),\qquad\mbox{if $s\le U$,}
\vspace*{5pt}\cr
\displaystyle C^{h,(0)}_{I^{(0)}_n} \biggl(\frac{U}{1-(V-U)} \biggr) + 1+ C^{h,(1)}_{I^{(1)}_n} \biggl( \frac{s-U}{V-U} \biggr),
\vspace*{3pt}\cr
\hspace*{127pt}\mbox{if $U<s\le V$,}
\vspace*{5pt}\cr
\displaystyle C^{h,(0)}_{I^{(0)}_n} \biggl(\frac{s-(V-U)}{1-(V-U)} \biggr),\qquad\mbox{if $ s> V$.}}
\end{equation}
Note that the vector $(I_n^{(0)}, I_n^{(1)})$ is a measurable function
of $(J_i)_{i = 1,\ldots, n}$, and thus independent of $U$,$V$,
$(C^{h,(0)}_n)_{n \geq1}$, and $(C^{h,(1)}_n)_{n \geq1}$. Moreover,
$I_n^{(0)}+\break I_n^{(1)}= n-1$ and, since the underlying fragmentation is a
Yule process, $I_n^{(1)}$ is uniform on $\{0, \ldots, n-1\}$.

As it has been observed by Curien and Le~Gall \cite{legalcu}, equation
(\ref{eqreccn2}) implies that the limit process $\mathscr H$
satisfies a
fixed-point equation in distribution:
let $\mathscr H^{(0)}$, $\mathscr H^{(1)}$ denote independent copies of
$\mathscr H$ such
that $(\mathscr H^{(0)}, \mathscr H^{(1)})$, $(U,V)$, $W$ are
independent, $(U,V)$
has density $2 \mathbf{1}_{ \{ 0 \leq u \leq v \leq1 \} }$ and $W$ is
uniformly
distributed on the unit interval. Then the process defined by, for
every $s\in[0,1]$,
%
\begin{equation}\label{eqrecH}
\cases{\displaystyle W^{1/3} \mathscr H^{ (0)} \biggl(\frac{s}{1-(V-U)} \biggr),\qquad\mbox{if $ s<U$,}
\vspace*{5pt}\cr
\displaystyle W^{1/3}\mathscr H^{(0)} \biggl( \frac{U}{1-(V-U)} \biggr) +(1-W)^{1/3} \mathscr H ^{(1)} \biggl( \frac{s-U}{V-U}\biggr),
\vspace*{3pt}\cr
\hspace*{148.5pt}\mbox{if $ U\le s<V$,}
\vspace*{5pt}\cr
\displaystyle W^{1/3} \mathscr H^{(0)} \biggl( \frac{s- (V-U)}{1-(V-U)} \biggr),\qquad\mbox{if $ s\ge V$,}}
\end{equation}
is distributed like the original process $\mathscr H$. Furthermore,
Curien and Le~Gall \cite{legalcu}, Section~8.1, prove that the limit
process $\mathscr H$ satisfies
%
\begin{equation}
\label{eqmeanhn} \mathbf{E} \bigl[\mathscr H(s)\bigr] = \kappa^{h}
\bigl(s(1-s)\bigr)^{1/2}
\end{equation}
for some constant $\kappa^{h} > 0$.

The techniques we have used in the case $\alpha= 2$ apply here, and
allow us to prove
the convergence of the dual tree $T_n^{h}$ in the Gromov--Hausdorff sense
(Theorem~\ref{thmdualtree0}). The limit process which is constructed
using our
functional approach is denoted by $H$, and is almost surely equal to
the process
$\mathscr H$ constructed in \cite{legalcu}.

%
\begin{theorem} \label{thmmain2}
As $n \to\infty$, we have
\[
\mathbf{E} \bigl[\bigl\| n^{-1/3} C_n^{h} - H
\bigr\|^m\bigr] \rightarrow0
\]
for all $m \in\mathbb{N}$. 
Moreover, for every $s\in[0,1]$ we have
\[
\mathbf{E} \bigl[H(s) \bigr] = \kappa^{h} \bigl(s(1-s)
\bigr)^{1/2}\qquad \mbox{where }\kappa ^{h} =
\frac{24}{\pi\Gamma(1/3)}.
\]
%
\end{theorem}
%
Again, Theorem~\ref{thmmain2} is much stronger than what is needed to
prove Theorem~\ref{thmdualtree0}, and implies convergence of all
moments of the height of the dual trees. Also, as in the self-similar
case discussed in Sections~\ref{seclimit} and~\ref{secconv}, the
leading constant $\kappa^{h}$ is identified using the asymptotic
expansion of $C_n^h$ at an independent random location~$\xi$.

%
\begin{theorem}\label{thmhunif}
Let $\xi$ be a uniform random variable, independent of all remaining
quantities. For every $n\ge1$, we have
\[
\mathbf{E} \bigl[C_n^{h}(\xi)\bigr] = \frac{\Gamma(n+4/3)}{\Gamma(4/3)n!}
-1 = \frac
{n^{1/3}}{\Gamma(4/3)} - 1 + O\bigl(n^{-2/3}\bigr).
\]
\end{theorem}


The remainder of the section is dedicated to the proofs of the previous
statements. However, since the techniques are essentially the same we
have already used in Sections~\ref{seclimit} and~\ref{secconv}, we
omit many details.

\subsection*{Mean at a uniform location} As in the case of the
self-similar lamination, the leading constant is identified using the
asymptotics
for $C_n^h$ at a uniformly random point $\xi$, independent of the
lamination. Write $Y_n^h:= C_n^h(\xi)$. Then we have
\[
Y_n^h \stackrel{d} {=} Y^h_{I_n^{(0)}} +
\mathbf{1}_{ \{ U \leq\xi<
V \} } \bigl(\widehat{Y^{h}}_{I_n^{(1)}} + 1
\bigr),
\]
where $(\widehat{Y^{h}}_n)_{n \geq1}$ is independent copy of
$(Y_n^h)_{n \geq1}$ and $(U,V)$, $(I_n^{(0)}$, $I_n^{(1)})$, $(Y_n^h)_{n
\geq1}$, $(\widehat{Y_n^h})_{n \geq1}$ are independent. Let now $\mu(n):= \mathbf{E} [Y_n^h]$. Taking expected values in the relation above yields
\[
\mu(n) = \frac{1} 3 + \frac{4}{3n} \sum
_{k =1}^{n-1} \mu(k).
\]
Elementary manipulations yield
\[
\frac{\mu(n)}{n+1} - \frac{ \mu(n-1)}{n} = \frac{1-2 \mu(n-1)}{3n(n+1)},
\]
%
so that
\[
\bigl(\mu(n)+1\bigr) = \bigl(\mu(n-1)+1\bigr) \biggl(1 + \frac{1}{3n}
\biggr),
\]
which implies the exact formula for $\mu(n)$. The expansion follows by
Stirling approximation.

The proofs of the remaining statements of Theorem~\ref{thmmain2} run
along very
similar lines as in the case $\alpha= 2$. In order to bound the
supremum of the process in $L^p$, we need to choose some
$p\ge1$ such that $\mathbf{E} [W^{p/3}+(1-W)^{p/3}]<1$.
Here, the coefficients $W^{1/3}, (1-W)^{1/3}$ are considerably larger than
$(1-(V-U))^\beta$, \mbox{$(V-U)^\beta$}. For this reason, $p=2$ is not
sufficient and
it is necessary to work with $p=4$.
However, note that the one-dimensional fixed-point equation for
$Y^h = H(\xi)$ arising from (\ref{eqrecH}) is
%
\begin{equation}
\label{eqhmeanrec} Y^h \stackrel{d} {=} W^{1/3} Y^h
+ \mathbf{1}_{ \{ U \leq\xi< V \} } (1-W)^{1/3} \widehat{Y^h}.
\end{equation}
Similarly to (\ref{eqfixunif}), $\widehat{Y^h}$ is distributed like
$Y^h$ and $Y^h, \widehat{Y^h}, W, \xi, (U,V)$ are independent.
Here, contraction in $L^2$ is guaranteed since the second coefficient
is substantially decreased by an independent Bernoulli variable with
success probability~$1/3$.

%
\begin{rem*}
Although the Brownian excursion has the same mean function
(see~\ref{eqmeanhn}), one easily verifies that $H=\mathscr H$ is not a
Brownian excursion,
and hence that $\mathcal T_H$ is not the Brownian continuum random tree (CRT).
We use the recursive equation (\ref{eqhmeanrec}) for $\mathbf{E}
[H(\xi)]$ to
show the
law of a standard Brownian excursion $(\mathbf e(s))_{s\in[0,1]}$ is not
invariant by
the transformation in (\ref{eqrecH}). If it were true $\mathbf{E}
[\mathbf e(\xi )^2]$ would
equal
%
\begin{equation}
\label{eqbrownianexc} \tfrac{10}{3} \mathrm{B} \bigl( \tfrac{4}3,
\tfrac{4}3 \bigr) \mathbf {E} \bigl[\mathbf e(\xi)\bigr]
\end{equation}
for a uniformly distributed random variable $\xi$ that is independent
of $\mathbf e$. However, as $\mathbf e(\xi)$ has the
standard Rayleigh distribution, we have $\mathbf{E} [\mathbf e(\xi)]
= \sqrt {\pi}/2$
and $\mathbf{E} [\mathbf e(\xi)^2] = 2$ which does not match the
value in
(\ref{eqbrownianexc}).
Further information about $H(\xi)$ may be obtained using the homogeneous
lamination process in continuous time: we find the following
characterization of
$H(\xi)$:
\[
H(\xi) \cdot E^{1/3} \stackrel{d} {=} E.
\]
Here, $E$ denotes an exponential random variable with mean one,
independent of $H$
and $\xi$.
As we do not draw further implications from this identity, we do not
give more
details about its proof here.
\end{rem*}

\subsection*{Construction of the limit} We need an additional
sequence of independent uniformly distributed random variables $(W_v)_{
v \in\mathcal T}$ that is independent of $(U_v,V_v)_{v \in\mathcal
T}$. Let $W = W_\varnothing$.
Define the operator $G^h\dvtx  A^+ \times(0,1) \times\mathcal C_0([0,1])^2
\to\mathcal C_0([0,1])$ by
\[
G^h[u,v,w;f_0,f_1](s)= \cases{
\displaystyle w^{1/3} f_0 \biggl( \frac{s}{1-(v-u)}\biggr),\qquad\mbox{if $s<U$,}
\vspace*{5pt}\cr
\displaystyle w^{1/3} f_0
\biggl( \frac{u}{1-(v-u)} \biggr) + (1-w)^{1/3} f_1 \biggl(
\frac{s-u}{v-u} \biggr),
\vspace*{3pt}\cr
\hspace*{126pt}\mbox{if $U\le s<V$,}
\vspace*{5pt}\cr
\displaystyle w^{1/3} f_0 \biggl( \frac{s- (v-u)}{1-(v-u)} \biggr),\qquad\mbox{if $s\ge V$.}}
\]
%
For every node $u\in\mathcal T$, let $H_0^{(u)}= \kappa^{h}
(s(1-s))^{1/2}$. Then define
recursively
\[
H_{n+1}^{(u)}=G^h\bigl[U_u,
V_u, W_u; H_n^{(u0)},
H_n^{(u1)}\bigr],
\]
%
or equivalently,
\[
H_{n+1}^{(u)}(s) = W_u^{1/3}
H_n^{(u0)}\bigl(K_0(s, U_u,V_u)
\bigr) + (1-W_u)^{1/3} H_n^{(u1)}
\bigl(K_1(s, U_u,V_u)\bigr), 
\]
%
where the functions $K_0, K_1$ are defined in (\ref{defK0K1}).
Finally, define $H_n=H_n^\varnothing$ to be the value observed at the
root of $\mathcal T$.
In order to prove uniform convergence of $H_n$ we investigate $\mathbf
{E} [\| H_{n+1} - H_n \|^4]$.
The analogue of (\ref{recmeansq}) involves three different kinds of
mixed terms and applying the $L^p$ inequality to any of them, we arrive at
%
\begin{equation}
\label{eql4} \Delta_n \leq q \Delta_{n-1} + 14
q' \bigl(\Theta_{n-1} ^{3/4} \Delta
_{n-1}^{1/4}+ \Theta_{n-1} ^{1/2}
\Delta_{n-1}^{1/2} + \Theta_{n-1} ^{1/4}
\Delta_{n-1}^{3/4}\bigr).
\end{equation}
Here, we used the abbreviations
\begin{eqnarray*}
\Theta_n &=& \mathbf{E} \bigl[\bigl(H_n(\xi)-
H_{n-1}(\xi)\bigr)^4 \bigr],
\\
\Delta_n&=& \mathbf{E} \bigl[\|H_{n+1}- H_n
\|^4 \bigr],
\\
q &=& \mathbf{E} \bigl[W^{4/3}\bigr] + \mathbf{E} \bigl[(1-W)^{4/3}
\bigr] = 6 / 7 < 1,
\\
q' &=& \mathbf{E} \bigl[W (1-W)^{1/3}\bigr] + \mathbf{E}
\bigl[W^{1/3} (1-W)\bigr] + \mathbf{E} \bigl[\bigl(W (1-W)
\bigr)^{2/3}\bigr].
\end{eqnarray*}
Note that $\Delta_n$ and $q$ correspond to the analogous terms in the
case $\alpha=2$ where we omit the superscript $h$ here. By the same arguments used to prove Lemma~\ref{lemonedim}, we
can show $\mathbf{E} [(H_n(\xi)- H_{n-1}(\xi))^2] \rightarrow0$
exponentially fast.
The following lemma whose proof is postponed is the necessary
additional ingredient proving uniform convergence.

\begin{lem} \label{leml4}Let $\bar q = \mathbf{E} [W^{2/3} + \mathbf
{1}_{ \{ U \leq\xi< V \} } (1-W)^{2/3}] = \frac{4} 5$.
For any $p \in\mathbb{N}$, $p \geq2$, we have 
%
\[
\mathbf{E} \bigl[\bigl|H_n(\xi)- H_{n-1}(\xi)\bigr|^p
\bigr] = O\bigl(\bar q^{n}\bigr).
\]
\end{lem}

Applying the lemma to the right-hand side of (\ref{eql4}) and using
the arguments
in the proof of Theorem~\ref{thmfixas} yields that $H_n$ converges
uniformly almost
surely and in~$L^4$ to a limit denoted by $H$. (We recall that
$H=\mathscr H$ with
probability one.)

\begin{pf*}{Proof of Lemma~\ref{leml4}}
As already mentioned,
\[
\mathbf{E} \bigl[\bigl(H_n(\xi)- H_{n-1}(\xi)
\bigr)^2 \bigr] = O\bigl(\bar q^n\bigr)
\]
can be verified by the same arguments as in the case $\alpha=2$. By
Jensen's inequality, we have  $\mathbf{E}  [|H_n(\xi)-
H_{n-1}(\xi)|  ] \leq
C_8\bar q^{n/2}$ for some constant $C_8 > 0$. For transferring the rate
to higher moments, we proceed by induction.
Let $p \geq3$ and assume the assertion is true for all $2 \leq j \leq
p-1$, that is, let $K_2, \ldots, K_{p-1}$ such that $\mathbf{E}
[|H_n(\xi)- H_{n-1}(\xi)|^j  ] \leq K_j \bar q^n$ for all $2
\leq j \leq p-1$. By the
observation in Lemma~\ref{lemonedim}, denoting by $\xi', \xi''$
independent random variables with uniform distribution that are
independent of all remaining quantities, we have
\begin{eqnarray*}
\Delta_{n}^{(p)}&:=& \mathbf{E}
\bigl[\bigl|H_n(\xi)- H_{n-1}(\xi)\bigr|^p \bigr]
\nonumber
\\
&=& \mathbf{E} \bigl[\bigl\llvert W^{1/3} \bigl(H^{(0)}_{n-1}
\bigl(\xi'\bigr)- H^{(0)}_{n-2}\bigl(
\xi'\bigr)\bigr)
\\
&&\hspace*{13pt}{} +\mathbf{1}_{ \{ U \leq\xi< V \} } (1-W)^{1/3}
\bigl(H^{(1)}_{n-1}\bigl(\xi{''}\bigr)-
H^{(1)}_{n-2}\bigl(\xi{''}\bigr)\bigr)
\bigr\rrvert ^p \bigr]
\\
& \leq&\mathbf{E} \bigl[W^{p/3} + \mathbf{1}_{ \{ U \leq\xi< V \}
}(1-W)^{p/3}
\bigr] \Delta_{n-1}^{(p)} + \sum_{\ell=1}^{p-1}
\pmatrix{p \cr\ell} K_\ell K_{p-\ell} \bar q^{(3/2) (n-1)}.
\end{eqnarray*}
%
Note that $\mathbf{E} [W^{p/3} + \mathbf{1}_{ \{ U \leq\xi< V \}
}(1-W)^{p/3}] = \frac{4}{p+3}
< \bar q$. Thus, by a simple induction on $n$, we obtain $\Delta
_{n}^{(p)} = O( \bar q^n)$.
\end{pf*}

\subsection*{The discrete process}
Let us give the coupling of the discrete process to its limit: for $u
\in\mathcal T$
and $n \in\mathbb{N}$, let $I_n^{(u)}$ be the number of fragments $v
\in
\mathcal{T}$ at
time $n$ with $v \in\mathcal T_u$. Here, $\mathcal T_u=\{u v\dvtx  v\in
\mathcal T
\}$ is the
set of nodes with prefix $u$. It is well known that the proportion
$I_n^{(u0)} /
\max( I_n^{(u)},1) $ converges almost surely to a uniform random
variable as
$n \to\infty$. We denote this limit by $W_u$. Then the sequence $(W_v)_{v
\in\mathcal T}$ is independent of the set $(U_v, V_v)_{v \in\mathcal
T} $
and consists of independent random variables having uniform distribution.
Based on these sets, for $u \in\mathcal T$, let $H^{(u)}$ be the process
constructed above.

The uniform convergence of $X_n^h:= C_n^h / (\Gamma(4/3) \mathbf{E}
[C_n^h(\xi )])$ to $H$ can be worked out analogously to
the case
$\alpha=2$ with similar modifications as for limit construction. The
only essential difference is the verification of
%
\begin{equation}
\label{l4oneconv} \lim_{n \rightarrow\infty} \bigl\llVert X_{I_n^{(0)}}^{h,(0)}(
\Psi) - H^{(0)}(\Psi) \bigr\rrVert _4 = 0
\end{equation}
instead of (\ref{l2one}) in $L^2$. There is no additional difficulty in
proving this $L^4$ convergence: first, convergence of the $L^2$
distance is obtained as in the proof of Theorem~\ref{thmconvunif} and second,
any absolute $p$th moment of $X_n^h(\xi)$ is bounded in $n$. This can
be shown by induction on $p$ using the result for $p =2$ as a base
case. Thus,
(\ref{l4oneconv}) follows by dominated convergence. We leave out the
remaining details of the proof, they should be clear from the arguments in
Section~\ref{secconv}.

%

\subsection*{Rates of convergence and almost sure convergence}
The rates for the convergence of the $L^p$ norm $\mathbf{E} [\|X_n^h -
H \|^p]$
are obtained by the same steps as in the case $\alpha= 2$.
First, note that given $W$, the random variable $I_n^{(0)}$ has
binomial distribution with parameters $n-1, W$. Thus, by the bound
(\ref{boundcher}), for any $k \geq1$,
\[
\mathbf{E} \bigl[\bigl| I_n^{(0)}/n - W \bigr|^{k/3}\bigr] =
O \bigl( n^{-k/6} \bigr),
\]
where we put $W:= W_\varnothing$.
Based on the latter bound for $k=2$, using the same methods as in the
case $\alpha= 2$, it follows that for any $\kappa< \frac{1} 6$:
\[
\mathbf{E} \bigl[\bigl(X_n^h(\xi) - H(\xi)
\bigr)^2\bigr] = O \bigl(n^{-\kappa} \bigr).
\]
Using the same arguments as in the proof of Proposition~\ref{propratetransprocess}, we can easily generalize the result to higher
moments. For any $m \geq1$ and $\kappa< \frac{1} {12}$, we have
\[
\mathbf{E} \bigl[\bigl|X_n^h(\xi) - H(\xi)\bigr|^m
\bigr] = O \bigl(n^{- \kappa m
} \bigr).
\]
As in the proof of Lemma~\ref{lemxnzL2} we can transfer the rate to
the process level. Based on an inequality similar to (\ref{eql4}), we
see that for any $\kappa< \frac{1} 6$,
\[
\mathbf{E} \bigl[\bigl\|X_n^h - H\bigr\|^4\bigr] = O
\bigl(n^{- \kappa} \bigr).
\]
Finally, analogously to Proposition~\ref{propratetransprocess}, we can
show that for any $m \geq1$ and \mbox{$\kappa< \frac{1}{24}$,}
\[
\mathbf{E} \bigl[\bigl\|X_n^h - H\bigr\|^m\bigr] = O
\bigl(n^{- \kappa m} \bigr).
\]
The almost sure convergence $\|X_n^h - H\| \to0$ follows as in the
case $\alpha=2$ by an application of the Borel--Cantelli lemma relying
on the
latter display for sufficiently large $m$. This implies almost sure
convergence of $n^{1/3} C_n^h$.
\section{Properties of the limit dual tree \texorpdfstring{$\mathcal{T}_Z$}{TZ}}\label{secpropdual}

In this section, we derive some important properties of the limit dual
tree $\mathcal T_Z$. The first set of properties are standard and characterize
the degrees in $\mathcal T_Z$. As in a discrete tree, for a real tree
$\mathcal T$
and $x\in\mathcal T$, the degree of $x$ in $\mathcal T$ is the number
of connected
components of $\mathcal T\setminus\{x\}$. Points of degree one are called
leaves. A real tree encoded by a continuous excursion $f\dvtx [0,1]\to
[0,\infty)$ comes with a natural probability measure, the push-forward
of Lebesgue measure on $[0,1]$ into the canonical projection $\pi_f\dvtx
[0,1]\to\mathcal T_f$.

%
\begin{prop}\label{proptree-prop}The real tree $(\mathcal T_Z,d_Z)$
is almost
surely compact, binary and has its mass concentrated on the leaves.
\end{prop}

\begin{pf}The compactness is an easy consequence of the fact that
$\mathcal T_Z$ is the image of $[0,1]$ under the canonical projection, which
is almost surely continuous for, for any $x,y\in[0,1]$,
\begin{eqnarray*}
d_Z\bigl(\pi_Z(x), \pi_Z(y)\bigr) & =&
Z(x)+Z(y)-2 \inf\bigl\{Z(s)\dvtx  x\wedge y\le s\le x\vee y\bigr\}
\\
& \le& 2 \sup\bigl\{\bigl|Z(s)-Z(t)\bigr|\dvtx  |s-t|\le|x-y|\bigr\} \to0
\end{eqnarray*}
as $|x-y|\to0$, since $Z$ is uniformly continuous.

Curien and Le~Gall \cite{legalcu}, Proposition 5.4, have proved that
the lamination encoded by $Z$ is almost surely a triangulation, which
implies that $\mathcal T_Z$ has maximal degree at most three with
probability one.

Finally, to prove that the mass measure is concentrated on the leaves,
it suffices to
show that for an independent uniform random variable $\xi$, $\pi
_Z(\xi
)\in\mathcal T_Z$ is
a leaf with probability one. By the rotational invariance, the degree
of $\pi_Z(\xi)$
is distributed like the degree of the root, say $\rho$. Now, since
$Z(s)>0$ for all
$s\in(0,1)$ with probability one (see \cite{legalcu}, proof of Corollary~5.3),
for all points $u,v\in\mathcal T_Z\setminus\{\rho\}$, the path from
$u$ to
$v$ in
$\mathcal T_Z$ does not go through $\rho$, so that $\mathcal
T_Z\setminus\{\rho\}$
has a single
connected component.
\end{pf}

We now look at the fractal dimension of $\mathcal T_Z$. For a metric space
$(X,d)$, we write $N(X,\delta)$ for the smallest size of a covering of
$X$ with balls of radius at most $\delta$. The box-counting dimension
is a law of large numbers for the size of coverings by balls of small
radius. More precisely, when
\[
\liminf_{\delta\downarrow0} \frac{\log N(X,\delta)}{\log(1/\delta)} =\limsup
_{\delta\downarrow0}\frac{\log N(X,\delta)}{\log(1/\delta)},
\]
the common value is called the Minkowski or box-counting dimension and
denote it by $\dim_M(X)$ \cite{Falconer1986,Falconer1990a}.

%
\begin{prop}\label{propbox-counting}Almost surely $\dim_M(\mathcal T
_Z)=1/\beta$.
\end{prop}

\begin{pf}
The upper bound is a simple consequence of the continuity properties of
the sample paths of $Z$. Since $Z$ is $\alpha$-H\"older continuous for
every $\alpha<\beta$ \cite{legalcu}, Theorem~1.1, there exists $C_\alpha<\infty$ almost surely such that for every
$x,y\in[0,1]$ one has
\[
\bigl|Z(x)-Z(y)\bigr|< C_\alpha|x-y|^\alpha.
\]
For $r>0$, fix $\delta>0$ such that $C_\alpha\delta^\alpha=r$. Let
$\pi
_Z$ denote the canonical projection from $[0,1]$ to $\mathcal T_Z$. The
collection $\{ B(\pi_Z(i\delta), 2 r)\dvtx  i=0,\dots, \lfloor\delta
^{-1}\rfloor\}$ is a covering of $\mathcal T_Z$. Indeed, for any point
$u\in
\mathcal T_Z$, there is $x\in[0,1]$ such that $\pi_Z(x)=u$ and
$i\delta\le
x<(i+1) \delta$ for some $i\le\lfloor\delta^{-1} \rfloor$ and by
definition of $\mathcal T_Z$
\begin{eqnarray*}
d_Z\bigl(u,\pi_Z(i\delta)\bigr) &\le& Z(x)+Z(i
\delta)-2 \inf\bigl\{Z(s)\dvtx  x\wedge i\delta\le s\le x \vee i\delta\bigr\}
\\
&\le&2 \sup\bigl\{\bigl|Z(s)-Z(t)\bigr|\dvtx  |s-t|\le\delta\bigr\}
\\
&<& 2 C_\alpha\cdot\delta^\alpha= 2 r.
\end{eqnarray*}
It follows immediately that $N(\mathcal T_Z, 2r)\le\delta^{-1}+1$ which
implies that
\[
\overline{\dim_M}(\mathcal T_Z):= \limsup
_{r\downarrow0} \frac
{\log N(\mathcal T
_Z,2r)}{\log(1/2r)} \le\frac{1} \alpha
\]
for any $\alpha<\beta$. Letting $\alpha\to\beta$ yields the upper bound.

For the lower bound, for every $r>0$ small enough, we exhibit a set of about
$r^{-1/\beta}$ points in $\mathcal T_Z$ in which any two points are at least
at distance $2r$ apart.
To this aim, we work directly with the fixed-point equation for $Z$,
that we can expand
in any way we like in order to exhibit a convenient partition of
$\mathcal T_Z$.
We rely on the fragmentation process underlying the construction of
$\mathcal T_Z$.

We use the standard embedding of the process of chord insertion in
continuous time, and modify slightly
the point of view of Section~\ref{subsecnotset}, in which chords
insertions may fail.
Let $\mathbf X(t)=(X_i(t)\dvtx  i\ge1)$ be the element of
$\ell^1_{\downarrow}:=\{(x_1,x_2,\ldots)\dvtx  x_1\ge x_2\ge\ldots, \sum_{i\ge1} x_1\le1 \}$ representing
the (ordered) sequence of fragment sizes at time $t$ in the
self-similar fragmentation of index $1$
and dislocation measure corresponding to the uniform binary split of
the mass \cite{Bertoin2006}.
[Then at the times of the split events $\tau_j$, $j\ge1$, $\mathbf
X(\tau
_j)$ is distributed like the
ordered sequence of masses of the fragments when $j$ chords have been inserted.]

Choosing one as the index of self-similarity turns out to be especially
convenient:
here, the number of chords $N_t$ at time $t$ is distributed like a
$\operatorname{Poisson}(t)$
random variable. Given $N_t$, $\mathbf X(t)$ is distributed like the ordered
sequence of spacings
created by $N_t$ uniformly random variables in $[0,1]$. So $N_t$ is
concentrated about $t$, and
the number of fragments of mass at least $1/t$ is roughly of order $t$.
More precisely,
writing $(\gamma_i)_{i\ge1}$ for a sequence of i.i.d. exponential
random variables with mean one,
then conditional on $N_t=n$, the collection of masses of the $n+1$
fragments (in random order)
is distributed like
\[
\biggl( \frac{\gamma_1}{\sum_{j=1}^{n+1} \gamma_j}, \ldots, \frac
{\gamma_{n+1}}{\sum_{j=1}^{n+1} \gamma_j} \biggr).
\]
Hence, for any $\delta>0$, $\varepsilon\in(0,1/\beta)$ and
$t=r^{-1/\beta
+\varepsilon}$, writing $A_\delta$ for the event that
$\#\{i\dvtx X_i(r^{-1/\beta+\varepsilon})\ge r^{1/\beta-\varepsilon}\}\ge
\delta
r^{-1/\beta+\varepsilon}$, we have
%
\begin{eqnarray}\label{eqpra}
{\mathbf P} \bigl(A_\delta^c \bigr)
& =& {\mathbf P} \Biggl(\# \Biggl\{i\le N_t\dvtx
\gamma_i \ge r^{1/\beta
-\varepsilon} \sum_{i=1}^{N_t}
\gamma_i \Biggr\}< \delta r^{-1/\beta
+\varepsilon} \Biggr)\nonumber
\\
&\le&{\mathbf P} \biggl(\sum_{1\le i\le t/2}
\mathbf{1}_{ \{ \gamma_i
\ge3 \} }< \delta r^{-1/\beta+\varepsilon} \biggr)
\nonumber\\[-8pt]\\[-8pt]
&&{}  +{\mathbf P}
\biggl(N_t\notin \biggl[\frac{t}2,2t \biggr] \biggr) + {
\mathbf P} \biggl(\sum_{1\le i \le2t} \gamma_i
\ge3t \biggr)
\nonumber
\\
&\le& r^{1/\beta-\varepsilon}\nonumber
\end{eqnarray}
for any $\delta<e^{-3}/2$ and for all $t$ large enough, by classical
large deviations bounds.
We now fix such a value for $\delta$ until the end of the proof.

For each $i\ge1$ and $t\ge0$, $X_i(t)$ is the mass of a subset
$S_i(t)$ of the tree $\mathcal T_Z$. Furthermore, each subset $S_i(t)$
is a connected
subset of $\mathcal T_Z$, and by the recursive
representation of $Z$, the subtree of $\mathcal T_Z$ induced on
$S_i(t)$ is
distributed like a
copy of $\mathcal T_Z$ in which all distances are multiplied by
$X_i(t)^\beta$
(by the fixed-point equation for $Z$).

We now always consider the set of fragments at time $r^{-1/\beta
+\varepsilon}$ for some $r>0$.
Let $I(r)$ be the set of indices corresponding to fragments (at time
$r^{-1/\beta+\varepsilon}$)
which contain a point at distance greater than $r$ from the rest of the tree.
Any covering of $\mathcal T_Z$ by balls of radius at most $r$ needs at least
one center per
element of $I(r)$, so that $N(\mathcal T_Z,r)\ge\# I(r)$. Note that although
the subset of
$\mathcal T_Z$ induced by $S_i$ is a tree, the degree $\deg(i)$ of
fragment $i$
(the number of connected components of $\mathcal T_Z\setminus S_i$) is
not bounded.
In particular, a large height does not guarantee the existence of a
point far from
$\mathcal T_Z\setminus S_i$. However, since the fragments are connected
subsets of a tree,
the average degree of the entire set of fragments is lower than two.
Also, writing $B$ for the event that the number of fragments at time
$r^{-1/\beta+\varepsilon}$
is at most $2r^{-1/\beta+\varepsilon}$, we have
%
\begin{equation}
\label{eqprb} {\mathbf P} \bigl(B^c \bigr)={\mathbf P} \bigl(
\operatorname {Poisson}\bigl(r^{-1/\beta+\varepsilon}\bigr)> 2 r^{-1/\beta+\varepsilon} \bigr)
\le r^{1/\beta-\varepsilon}
\end{equation}
for all $r$ small enough.
If both $A_\delta$ and $B$ occur then the average degree of the set of
fragments of mass at least
$r^{1/\beta+\varepsilon}$ is at most $4\delta^{-1}$. This implies that on
$A_\delta\cap B$
%
\begin{equation}
\label{eqfraclbnumber} \# \bigl\{i\dvtx  X_i\bigl(r^{-1/\beta+\varepsilon}\bigr)\ge
r^{1/\beta-\varepsilon}, \deg (i)\le8\delta^{-1} \bigr\}\ge\frac\delta2
r^{-1/\beta+\varepsilon}.
\end{equation}

For a given fragment $S_i$ of degree at most $8\delta^{-1}$, the
existence of a path of length at least
$16 \delta^{-1} r$ within the fragment implies that there is
a point $u\in S_i$ at distance at least $r$ from $\mathcal T_Z\setminus S_i$.
So, writing $Z_i$ for the height process within $S_i$, we have
\[
\# I(r)\ge\# \bigl\{i\dvtx  X_i\bigl(r^{-1/\beta+\varepsilon}\bigr)\ge
r^{1/\beta
-\varepsilon}, \deg(i)\le8 \delta^{-1}, \|Z_i\|>16
\delta^{-1}r \bigr\}.
\]
Note that the scaling property of $Z$ implies that, for $i$ such that
$X_i\ge r^{1/\beta-\varepsilon}$,
%
\begin{equation}
\label{eqfraclbproba} 
{\mathbf P}\bigl(\|Z_i\|> 16
\delta^{-1}r \bigr) \ge{\mathbf P}\bigl(\|Z\| >16
\delta^{-1} r^{\varepsilon\beta}\bigr)>1/2
\end{equation}
for all $r$ small enough since $\|Z\|>0$ with probability one. Now,
given the sequence of masses $X_i(r^{-1/\beta+\varepsilon})$ at time
$r^{-1/\beta+\varepsilon}$,
the \emph{internal} structure of the fragments
and in particular the processes $Z_i$ are independent.
So, for every $r>0$ small enough, by (\ref{eqfraclbnumber}) and
(\ref{eqfraclbproba}), on the event $A_\delta\cap B$ the random
variable $\#I(r)$ dominates
a binomial random variable with parameters $\frac\delta2 r^{-1/\beta
+\varepsilon}$ and $1/2$.
It follows that, for all $r$ small enough,
%
\begin{equation}
\label{eqgivenab} \mathbf{P} \biggl(\# I(r) \le\frac\delta8 r^{-1/\beta
+\varepsilon} \bigg|
A_\delta, B \biggr) \le\frac{16} \delta r^{1/\beta-\varepsilon}
\end{equation}
by Chebyshev's inequality. Putting (\ref{eqpra}), (\ref{eqprb}) and
(\ref{eqgivenab}) together,
the Borel--Cantelli lemma implies that almost surely for all
$r$ small enough along the sequence $r=2^{-n}$, $n\ge0$, we have
\[
\# I(r) \ge\frac\delta{8} r^{-1/\beta+\varepsilon}.
\]
Finally, since for $r\in[2^{-(n+1)}, 2^{-n})$ we have $N(\mathcal
T_Z,r)\ge
N(\mathcal T_Z,2^{-n})$, we obtain
\[
\underline{\dim_M}(\mathcal T_Z):=\liminf
_{r\downarrow0} \frac
{\log N(\mathcal T
_Z,r)}{\log(1/r)} \ge\liminf_{r\downarrow0}
\frac{\log\#
I(r)}{\log
(1/r)} \ge\frac{1}\beta-\varepsilon,
\]
which proves the desired lower bound since $\varepsilon>0$ was arbitrary.
\end{pf}

\setcounter{equation}{0}
\begin{appendix}
\section*{Appendix: Expected value at a uniform location: Proof of Theorem~\lowercase{\protect\ref{thmunif}}}
\label{secmean} \label{secunif} \label{secuniform}
Our approach is to first find a \emph{explicit} formula for the mean,
then extract precise asymptotics via complex analytic methods.

\subsection*{A recurrence relation}
Under the model of Section~\ref{secconv}, let $Y_n:= C_n(\xi)$. We
start with the derivation of a recurrence relation for the quantity
$\mu
(n):=\mathbf{E} [Y_n]$. We have,
by construction,
%
\begin{equation}
\label{recunif} Y_n \stackrel{d} {=} \mathbf{1}_{ \{ \xi\in S^{(0)} \cap\mathscr C \} }
Y^{(0)}_{I_n^{(0)}} + \mathbf{1}_{ \{ \xi\in S^{(1)} \cap\mathscr C
\} }
\bigl[1+Y^{(0)}_{I_n^{(0)}} + Y^{(1)}_{I_n^{(1)}}
\bigr], \qquad n \geq1,
\end{equation}
with independent copies $(Y^{(0)}_n)_{n \geq0}, (Y^{(1)}_n)_{n \geq
0}$ of $(Y_n)_{n \geq0}$, independent of $(\xi, U,V, I_n^{(0)}, I_n^{(1)})$.
For the definition of $S^{(0)}$ and $S^{(1)}$, see the beginning of
Section~\ref{subsecnotset}. Observe that, given
$\{ \xi\in S^{(0)} \cap\mathscr C\}$, $\ell^{(0)}$, the mass of $S^{(0)}$
and defined in (\ref{deflv}), is distributed as the maximum of three
independent uniform random variables.
Hence, it has the Beta$(3,1)$ distribution with density $3u^2$.
Additionally, both $\ell^{(0)}$ given $\{\xi\in S^{(1)} \cap\mathscr
C\}$,
and $\ell^{(1)}$ given $\{\xi\in S^{(1)} \cap\mathscr C\}$ are distributed
as the second smallest of three independent
uniform random variable.
Therefore, both have Beta$(2,2)$ distribution with density $6u(1-u)$.

Equation (\ref{recunif}) implies that $\mu(n):= \mathbf{E} [Y_n]$ satisfies
the following recurrence relation for all $n\ge1$:
%
\begin{eqnarray} \label{recdumm}
\mu(n) &= & \frac{1} 3 + \frac{2} 3 \sum
_{k=1}^{n-1} \mu(k) \cdot \bigl[ {\mathbf P}\bigl(
\operatorname{Bin}\bigl(n-1,B_{3,1}^2\bigr) = k\bigr)\nonumber
\\
&&\hspace*{78pt} {}  + {\mathbf P}\bigl(\operatorname{Bin}\bigl(n-1,B_{2,2}^2\bigr)
= k\bigr) \bigr]
\\
&= & \frac{1} 3 + \sum
_{k=1}^{n-1} \mu(k) \pi_{n,k},\nonumber
\end{eqnarray}
where we wrote
\[
\pi_{n,k}= \pmatrix{n-1 \cr k} \bigl[ 2\mathrm{B}(k+1, n-k) -
\mathrm{B}(k+3/2, n-k) \bigr].
\]
%

\subsection*{An exact expression for \texorpdfstring{$\mu(n)$}{mu(n)}}
Although it is linear, the recurrence relation~(\ref{recdumm})
involves an unbounded number of terms. We adopt an approach using
generating functions which are particularly adapted. Define $M(z):=
\sum_{n \geq1} \mu(n) z^n$. Since $\mu(n)\le n$ the convergence radius
of $M(z)$ is
exactly one.
In a different but related case, Flajolet et al. \cite{FlGoPuRo1993}
derived a differential equation for $M(z)$ from the recursion (\ref{recdumm}), which is explicitly solvable.
In our case, this does not seem possible, and we follow ideas used by
Chern and Hwang in \cite{ChHw2003} which rely on a differential equation
for the Euler transform of $M(z)$. We start by defining the binomial
and Euler transforms (see \cite{Knuth1973b}, p.~137, or \cite{SeFl1996}, pp.~105--106).
Given a sequence of real numbers $(a(n), n\ge0)$, the binomial
transform $a^*$ of~$a$ is defined by
\[
a^*(n) = \sum_{k = 0}^n \pmatrix{n \cr k}
(-1)^k a(k).
\]
The sequences $a$ and $a^*$ are then dual in the sense that $a=(a^*)^*$.
The binomial transform of a sequence of numbers is related to the Euler
transform of its generating function. Given a function $f\dvtx  \mathbb{C}
\to\mathbb{C}$, analytic in a neighborhood of the origin, define its
Euler transform $f^*$ by
\[
f^*(z) = \frac{1}{1-z} f \biggl( \frac{z}{z-1} \biggr).
\]
Note that $f^*$ is also analytic in some neighborhood of zero.
The function $f^*$ has the crucial property that
its coefficients are given by the binomial transform of the
coefficients of $f$. In particular, $M^*(z) = \sum_{n \geq1} \mu
^*(n) z^n$.

The basic observation is that $\pi_{n,k}$ may be expressed in terms
which relate to binomial transforms. We give an expression for $\pi
_{n,k}$ that may look slightly artificial, but actually exploits the
structure and make the binomial transform explicit:
%
\begin{equation}
\label{eqpink} \pi_{n,k}=\sum_{j = k}^{n-1}
\pmatrix{n-1 \cr j} \pmatrix{j \cr k} (-1)^{j+k} \biggl[\frac{2}{j+1} -
\frac{1} {j+3/2} \biggr].
\end{equation}
To see that (\ref{eqpink}) indeed holds, observe that we have
\begin{eqnarray*}
\pmatrix{n-1 \cr k} \mathrm{B}(k+1, n-k) &= & \pmatrix{n-1 \cr k} \int
_0^1 x^{k} (1-x)^{n-k-1} \,dx
\\
&= & \pmatrix{n-1 \cr k} \sum_{j=0}^{n-k-1}
(-1)^j \pmatrix{n-k-1 \cr j} \int_0^1
x^{k+j } \,dx
\\
&= & \pmatrix{n-1 \cr k} \sum_{j = k}^{n-1}
(-1)^{j+k} \pmatrix{n-k-1 \cr j-k} \frac{1}{j+1}
\\
&= & \sum_{j = k}^{n-1} \pmatrix{n-1 \cr j} \pmatrix{j \cr k} \frac{(-1)^{j+k} }{j+1}.
\end{eqnarray*}
The expression for the second term is obtained in the same way:
\begin{eqnarray*}
\pmatrix{n-1 \cr k} \mathrm{B}(k+3/2, n-k) &= & \sum_{j = k}^{n-1}
\pmatrix{n-1 \cr j} \pmatrix{j \cr k} \frac
{(-1)^{j+k}} {j+3/2},
\end{eqnarray*}
which proves the identify in (\ref{eqpink}).

Using (\ref{eqpink}) and the binomial transform immediately yields
%
\begin{eqnarray}
\sum_{k=1}^{n-1} \mu(k) \pi_{n,k}
&= & \sum
_{j=1}^{n-1} \pmatrix{n-1 \cr j} (-1)^j
\biggl[\frac{2}{j+1} - \frac{1} {j+3/2} \biggr]\mu^*(j) \label{verbin}.
\end{eqnarray}
Having in mind a second application of the binomial transform, this
leads us to define
%
\begin{equation}
\label{defS} S(z) = \sum_{n \geq1} \biggl[
\frac{2} {n+1}-\frac{1} {n+3/2} \biggr] \mu^*(n) z^n.
\end{equation}
The function $S(z)$ is crucial for our approach. 
We will exhibit two connections between $S(z)$ and $M(z)$ which will
finally imply the preannounced differential equation for the Euler
transform $M^*$ of $M$.

Since $M^*$ is analytic at the origin, the formal term-wise integration
makes sense in a neighborhood of zero. Thus, 
(\ref{defS}) yields
%
\begin{equation}
\label{connect1} \frac{d}{dz} \bigl(z^{3/2} S(z)\bigr) =
z^{1/2} M^*(z) + z^{-1/2} \int_0^z
M^*(u)\,du.
\end{equation}
Furthermore, the recurrence relation for $\mu(n)$ in (\ref{recdumm}),
together with (\ref{verbin}) and~(\ref{defS}) imply that
\begin{eqnarray*}
M(z) 
&= & \frac{1}{3} \frac{z}{1-z} +
\frac{z}{1-z} S \biggl( \frac
{z}{z-1} \biggr).
\end{eqnarray*}
In particular, for $w = z / (z-1)$
%
\begin{equation}
\label{connect2} w^{1/2} (w-1) M^*(w) = \tfrac{1}3
w^{3/2} + w^{3/2} S(w).
\end{equation}

Combining (\ref{connect1}) and (\ref{connect2}) provides an
integro-differential equation for $M^*$.
\[
2w(w-1) \frac{d}{dw} M^*(w) + (w-1)M^*(w) - 2 \int_0^w
M^*(u)\,du = w.
\]
Direct comparison of the $n$th coefficient of both sides gives a
one-term recursion for the Binomial transforms $\mu^*(n)$ for $n \geq
2$, which implies
%
\[
\mu^*(n) = -\frac{1} 3 \prod_{j=2}^n
\frac{2j^2 -j-2}{j(2j+1)}.
\]
Finally, using the duality of binomial transforms, we have just proved
the first assertion of Theorem~\ref{thmunif}.

\subsection*{Asymptotic estimate for \texorpdfstring{$\mu(n)$}{mu(n)}}
The representation
for $\mu(n)=\mathbf{E} [C_n(\xi)]$ in~(\ref{eqmun}) involves an alternate
series and is delicate to evaluate asymptotically: although some terms
are exponentially large in absolute value, we know from the
combinatorial setting that $\mu(n)\le n$, so that an approach that
focuses on the individual terms is bound to be rather intricate. For
the asymptotic expansion, we will use methods based on N\"orlund--Rice
integrals which is standard for the calculus of finite differences
\cite
{Norlund1924a,FlSe1995a,FlSe2009a}. We now show the asymptotic
expression in (\ref{constunif}).

First note that, writing $\gamma=(1+\sqrt{17})/4$ and $\bar\gamma
=(1-\sqrt{17})/4$, we have
\begin{eqnarray*}
&& \frac{1} 3 \cdot\prod_{j=2}^k
\frac{2j^2-j-2}{j(2j+1)}
\\
&&\qquad = \frac{1} 3 \cdot\prod_{j=2}^k
\frac{(j-\gamma)(j-\bar\gamma)}{j
(j+1/2)}
\\
&&\qquad =\frac{1} 3 \cdot\frac{\Gamma(k+1-\gamma)}{\Gamma(2-\gamma)} \cdot \frac{\Gamma(k+2-\bar\gamma)}{\Gamma(2-\bar\gamma)} \cdot
\frac{1} {\Gamma(k+1)} \cdot\frac{\Gamma(5/2)}{\Gamma(k+3/2)}=: f(k), 
\end{eqnarray*}
where it is understood that the last line above defines the function
$f$ at the integer $k\ge2$. Thus,
\[
\mu(n)=\sum_{k=1}^n \pmatrix{n \cr k}
(-1)^{k+1} f(k).
\]
The definition of $f$ extends analytically to complex values $z$ for
which none of the arguments of the Gamma functions takes a value in the
nonpositive integers. So, we have
so that writing
%
\begin{eqnarray}
\mu(n)&=&\sum_{k=1}^n \pmatrix{n \cr k}
(-1)^{k+1} f(k)\nonumber
\\
\eqntext{\displaystyle\mbox{with } f(z) = \frac{\sqrt{\pi} \Gamma(z- \gamma+1)\Gamma(z - \bar
\gamma+1)}{4 \Gamma(z+1) \Gamma(z+3/2) \Gamma(2-\gamma) \Gamma
(2-\bar \gamma)}.}
\end{eqnarray}
%
We may apply the results in Section~2 of \cite{FlSe1995a} which yield
the following integral representation for $\mu(n)$
%
\begin{equation}
\label{rice} \mu(n) = \frac{1}{2 \pi i} \int_{\mathscr C}
\frac{(-1)^{n+1} n!}{z
(z-1) \cdots(z-n)} f(z) \,dz,
\end{equation}
where $\mathscr C$ is any positive contour encircling the segment
$[1,n]$, which lies in the domain of analyticity of $f$ and avoids the
nonnegative integers. We take $\mathscr C$ to be the contour consisting
of the vertical line $c+i \ensuremath{\mathbb{R}}$ and loops around
the segment $[1,n]$ from
$c-i\infty$ to $c+i\infty$.
Observe that the integrand in (\ref{rice}) has singularities in the set
$\Re(z) < 1$ at $\gamma-1 - \ell$ and $\bar\gamma-1 - \ell$ for
$\ell\geq0$ and zero.
All these singularities are simple poles. We would like to shift the
vertical portion of the contour integration towards the left in order
to peel off the first pole we meet, thus extracting the main asymptotic
contribution.

It follows that for $0 < c' < \gamma-1$, and $\mathscr C'$ the shift
of the contour $\mathscr C$ which has its vertical part along
$c'+i\ensuremath{\mathbb{R}}$,
we have
%
\begin{eqnarray}\label{eqmunrice}
\mu(n) 
&=& \frac{(-1)^n n! \cdot\operatorname{Res}(f;\gamma-1)}{(\gamma
-1)\cdots
(\gamma- n-1)} 
\nonumber\\[-8pt]\\[-8pt]
&&{}+ \frac{1}{2 \pi i} \int_{c'-i\infty}^{c'+i\infty}
\frac{(-1)^{n} n!}{z
(z-1) \cdots(z-n)} f(z) \,dz,\nonumber
\end{eqnarray}
because, by Stirling formula, $|f(z)|=O(|z|^{-1})$ as $|z|\to\infty$
inside the half-plane $\{\Re(z)\ge c'\}$.
We first estimate the main contribution, which comes from the term
involving the residue $\operatorname{Res}(f;\gamma-1)$ of $f$ at $z =
\gamma
-1$. Using the fact that $(\gamma-1)(1- \bar\gamma) = 1/2$, we
easily obtain
\begin{eqnarray*}
\frac{(-1)^n n! \cdot\operatorname{Res}(f;\gamma-1)}{(\gamma
-1)\cdots(\gamma
- n-1)} 
&= & c
n^{\gamma-1}+O\bigl(n^{\gamma-2}\bigr),
\end{eqnarray*}
where $c$ is the constant defined in (\ref{constunif}).

In the same way, one proves that the remaining term in (\ref{eqmunrice}) is $O(1)$ as $n\to\infty$: one first easily obtains that
\begin{eqnarray*}
&& \biggl\llvert \frac{1}{2 \pi i} \int_{c' - i \infty}^{c'+ i \infty}
\frac
{(-1)^n n!}{z (z-1) \cdots(z-n)} f(z) \,dz\biggr\rrvert
\\
&&\qquad \le\frac{1} {2\pi} \int
_{c'-i\infty}^{c'+i\infty} \frac{\Gamma(n+1)
|\Gamma(1-c')|}{|\Gamma(n-c'+1)|} \frac{|f(z)|}{|z|}
|\,dz|
\\
&&\qquad = O\bigl(n^{c'}\bigr)
\end{eqnarray*}
by Stirling's formula. To prove a bound of $O(1)$, one shifts again the
vertical line to the left at and peels off the next residue,\vspace*{1.5pt} which
happens to be at $z=0$: the residue itself contributes $O(1)$ and the
remainder $O(n^{c''})$ for $c''<0$. We omit the details.
\end{appendix}



%

\printaddresses

\end{document}